\documentclass[12pt]{amsart}

\usepackage[active]{srcltx}
\usepackage{calc,amssymb,amsthm,amsmath,amscd, eucal,ulem,stmaryrd}
\usepackage{alltt}
\usepackage[left=1.35in,top=1.25in,right=1.35in,bottom=1.25in]{geometry}
\usepackage{mathtools}
\usepackage{soul}
\usepackage{hyperref}
\usepackage{mathrsfs}
\usepackage{xspace}
\usepackage[dvipsnames,svgnames]{xcolor}

\makeatletter
\def\@tocline#1#2#3#4#5#6#7{\relax
  \ifnum #1>\c@tocdepth 
  \else
    \par \addpenalty\@secpenalty\addvspace{#2}%
    \begingroup \hyphenpenalty\@M
    \@ifempty{#4}{%
      \@tempdima\csname r@tocindent\number#1\endcsname\relax
    }{%
      \@tempdima#4\relax
    }%
    \parindent\z@ \leftskip#3\relax \advance\leftskip\@tempdima\relax
    \rightskip\@pnumwidth plus4em \parfillskip-\@pnumwidth
    #5\leavevmode\hskip-\@tempdima
      \ifcase #1
       \or\or \hskip 1em \or \hskip 2em \else \hskip 3em \fi%
      #6\nobreak\relax
    \hfill\hbox to\@pnumwidth{\@tocpagenum{#7}}\par
    \nobreak
    \endgroup
  \fi}
\makeatother

\input{kmacros3.sty}
\input{xy}
\xyoption{all}
\usepackage{mabliautoref}
\usepackage{bm}

\usepackage{pifont}

\DeclareSymbolFont{rsfs}{OMS}{rsfs}{m}{n}
\DeclareSymbolFontAlphabet{\scr}{rsfs}

\DeclareMathOperator{\arc}{arc}

\usepackage{relsize}
\usepackage[bbgreekl]{mathbbol}
\usepackage{amsfonts}
\DeclareSymbolFontAlphabet{\mathbb}{AMSb} 
\DeclareSymbolFontAlphabet{\mathbbl}{bbold}
\newcommand{\Prism}{{\mathlarger{\mathbbl{\Delta}}}}

\renewcommand{\m}{\mathfrak{m}}
\renewcommand{\n}{\mathfrak{n}}
\renewcommand{\fram}{\mathfrak{m}}
\numberwithin{equation}{theorem}

\usepackage{fullpage}

\usepackage{setspace}

\usepackage{enumerate}
\usepackage{enumitem}
\usepackage{graphicx}
\usepackage[all,cmtip]{xy}

\usepackage{upgreek}

\usepackage{bbding}
\usepackage{verbatim}

\renewcommand{\O}{\mathcal O}

\newcommand{\perf}{\textnormal{perf}}
\newcommand{\perfd}{\textnormal{perfd}}
\newcommand{\WCart}{\mathrm{WCart}}
\newcommand{\HT}{\mathrm{HT}}
\newcommand{\HPrism}{\mathscr{H}_{\Prism}}

\newcommand{\absperfdinj}{\text{lim-perfectoid injective}\xspace}
\newcommand{\claperfdinj}{perfectoid injective\xspace}
\newcommand{\absperfdpure}{lim-perfectoid pure\xspace}
\newcommand{\claperfdpure}{perfectoid pure\xspace}
\newcommand{\classicalperfdpure}{{\xspace}perfectoid pure\xspace}

\renewcommand{\Cart}{\mathrm{Cart}}

\begin{document}

\title{Perfectoid pure singularities}
\author{Bhargav Bhatt, Linquan Ma, Zsolt Patakfalvi, Karl Schwede, Kevin Tucker, Joe Waldron, Jakub Witaszek}
\address{Institute for Advanced Study and Princeton University}
\email{bhargav.bhatt@gmail.com}
\address{Department of Mathematics, Purdue University, West Lafayette, IN 47907, USA}
\email{ma326@purdue.edu}
\address{\'Ecole Polytechnique F\'ed\'erale de Lausanne (EPFL), MA C3 635, Station 8, 1015 Lausanne, Switzerland}
\email{zsolt.patakfalvi@epfl.ch}
\address{Department of Mathematics, University of Utah, Salt Lake City, UT 84112, USA}
\email{schwede@math.utah.edu}
\address{Department of Mathematics, University of Illinois at Chicago, Chicago, IL 60607, USA}
\email{kftucker@uic.edu}
\address{Department of Mathematics, Michigan State University, East Lansing, MI 48824, USA}
\email{waldro51@msu.edu}
\address{Northwestern University, Department of Mathematics, Evanston, IL 60208, USA
}
\email{jakub.witaszek@northwestern.edu}

\begin{abstract}
  Fix a prime number $p$.
  Inspired by the notion of $F$-pure or $F$-split singularities, we study the condition that a Noetherian ring with $p$ in its Jacobson radical is pure inside some perfectoid (classical) ring, a condition we call \emph{\claperfdpure}.  We also study a related a priori weaker condition which asks that $R$ is pure in its absolute perfectoidization, a condition we call \emph{\absperfdpure}.  We show that both these notions coincide when $R$ is LCI.  Mixed characteristic analogs of $F$-injective and Du Bois singularities are also explored.  We study these notions of singularity, proving that they are weakly normal and that they are Du Bois after inverting $p$.  We also explore the behavior of \claperfdpure singularities under finite covers and their relation to log canonical singularities.  Finally, we prove an inversion of adjunction result in the LCI setting, and use it to prove that many common examples are perfectoid pure.
\end{abstract}

\maketitle
\setcounter{tocdepth}{1}
\tableofcontents

\section{Introduction}

Suppose $R$ is a Noetherian ring of characteristic $p > 0$.  We say that $R$ is \emph{$F$-split} if the Frobenius map $R \to F_* R$ is split as a map of $R$-modules.  Under moderate hypotheses, this is equivalent to the map $R \to F_* R$ being \emph{pure} (aka universally injective) in which case we say that $R$ is $F$-pure.  Every regular ring is $F$-pure, and many important singular rings are also $F$-pure.  Because of the nice properties that $F$-split and $F$-pure rings possess, the condition has been intensively studied since the 1970s, see for instance \cite{HochsterRobertsFrobeniusLocalCohomology,MehtaRamanathanFrobeniusSplittingAndCohomologyVanishing,BrionKumarFrobeniusSplitting}.

For Noetherian rings, $F$-pure singularities are quite closely related to log canonical singularities, \cf  \cite{HaraWatanabeFRegFPure,MustataSrinivasOrdinary,BhattSchwedeTakagiweakordinaryconjectureandFsingularity}, a central class of singularities in the minimal model program and moduli theory in characteristic zero.  The goal of this paper is to study a variant of these notions in mixed characteristic.

Note that $R$ is $F$-pure if and only if the map to the perfection $R \to R_{\perf} := \bigcup_e R^{1/p^e} = \colim_e F^e_* R$ is {pure}. In mixed characteristic, while we no longer have Frobenius, thanks to \cite{ScholzePerfectoidspaces,BhattScholzepPrismaticCohomology} we have nice analogs of perfect rings and perfection, namely, \emph{perfectoid rings} and the \emph{perfectoidization}. We define $R$ to be \emph{\claperfdpure} if there exists a perfectoid $R$-algebra $B$ and an $R$-algebra homomorphism such that $R \to B$
is {pure}.  This $B$ need not be arbitrary, in fact, when $(R,\m)$ is a Noetherian complete local domain with perfect residue field $k$, $R$ is \claperfdpure if and only if $$R \to B := (R \otimes_A A[p^{1/p^{\infty}}, x_2^{1/p^{\infty}}, \dots, x_d^{1/p^{\infty}}]^{\wedge_p})_{\perfd}$$ is pure, where $A = W(k)\llbracket x_2, \dots, x_d\rrbracket$ is a Noether-Cohen normalization of $R$ (i.e., $A\to R$ makes $R$ into a finite $A$-module) and $(-)_{\perfd}$ is the perfectoidization functor introduced in \cite[Sections 7 and 8]{BhattScholzepPrismaticCohomology}, see \autoref{lem:RAperfdsuffices}.

Related to $F$-pure singularities in characteristic $p > 0$ and log canonical singularities in characteristic zero are the notions of $F$-injective and Du Bois singularities, respectively \cite{FedderFPureRat,DuBoisMain,SteenBrinkDuBoisReview}.
Note that $F$-pure $\Rightarrow$ $F$-injective, and log canonical $\Rightarrow$ Du Bois \cite{KollarKovacsLCImpliesDB}, while the converse implications hold when the ring is quasi-Gorenstein. 
We say that a Noetherian local ring $(R,\fram)$ of characteristic $p>0$ is $F$-injective if 
\[H^i_{\fram}(R)\to H^i_{\fram}(R_{\perf})\] is injective for all $i$.  On the other hand, if $(R,\fram)$ is essentially of finite type over $\mathbb{C}$, $R$ is Du Bois if 
\[H^i_{\fram}(R)\to H^i_{\fram}(\underline{\Omega}_R^0)\] is injective for all $i$. Note, this map is always surjective, \cite{KovacsDuBoisLC1,KovacsSchwedeDBDeforms}, and so injectivity is equivalent to it being an isomorphism, which is in turn equivalent to the usual definition of Du Bois $(R \cong \DuBois{R}$) by duality, \cf \cite{KovacsDuBoisLC1,BhattSchwedeTakagiweakordinaryconjectureandFsingularity,GodfreyMurayamaPureOfDuBois}.
Inspired by these definitions, we say that a Noetherian local ring $(R, \fram)$ is \claperfdinj if 
\[
    H^i_{\fram}(R) \to H^i_{\fram}(B)
\]
is injective for all $i$, for some perfectoid $R$-algebra $B$. Once again, it suffices to check this on certain specific $B$ by \autoref{lem:RAperfdsuffices}.

The two equicharacteristic definitions above can be combined into a single statement: over $\bC$ we have that $\underline{\Omega}_R^0=\myR\Gamma_h(\Spec(R),\cO)$ by \cite{LeeLocalAcyclicFibrationsAndDeRham, HuberJorderDifferentialFormsHTopology}, and over a field of characteristic $p>0$, $R_{\perf}=\myR\Gamma_h(\Spec(R),\cO)$ by \cite{BhattSchwedeTakagiweakordinaryconjectureandFsingularity}, where $\myR\Gamma_h$ denotes derived global sections from the $h$-topology. Therefore we also  
develop an a priori weaker notion:
we say that $R$ is \emph{\absperfdinj} if 
\[
    H^i_{\fram}(R) \to H^i_{\fram}(R_{\perfd})
\]
injects for all $i \geq 0$, where 
$R_{\perfd} = \myR \Gamma_{\arc}(\Spf(R), \cO)$ is the absolute perfectoidization, see \cite{BhattScholzepPrismaticCohomology}. Here the arc-topology \cite{BhattMathewArc} is a Grothendieck topology more suited to the $p$-complete and perfectoid setting, but which is equivalent to the $h$-topology on Noetherian schemes. We show that $R_\perfd$ is the inverse limit of all perfectoid rings $B$ that admit a map from $R$, i.e., $R_\perfd = \varprojlim_{R\to B}B$ where the limit is taken in $\widehat{D}(\mathbf{Z}_p)$, see \autoref{prop:universal_property_perfd}.

We also say that $R$ is \emph{\absperfdpure} if the map 
  $  R \to R_{\perfd}$
is pure.  Note $R_{\perfd}$ is a derived object and not a classical ring (as indeed was $\underline{\Omega}_R^0$), and so some care must be taken to define ``pure'', see \autoref{subsec.PureMapsInDR}.  
If $R$ is a ring of characteristic $p>0$, we have that $R_{\perf}=R_{\perfd}$, thus it follows that our definitions agree with the usual characteristic $p$ definitions in this case.  We expect that perfectoid pure and lim-perfectoid pure (respectively, perfectoid injective and lim-perfectoid injective) are equivalent in general. We can show this when all local rings of $R$ are complete intersections (i.e., $R$ is LCI), in fact, in this case, all four notions agree. 

\begin{theoremA*}[{\autoref{cor:LCIallequivalent}}]
Let $R$ be a Noetherian ring with $p$ in its Jacobson radical. Suppose $R$ is LCI. Then $R$ being perfectoid pure, \absperfdpure, perfectoid injective, and \absperfdinj are all equivalent.
\end{theoremA*}

Furthermore, in the LCI case, we can show an inversion of adjunction-type result. Note that, thanks to Theorem A above, in the statement of the next result we can replace \claperfdinj by any of the other three notions. 

\begin{theoremB*}[{\autoref{thm: inv of adj}}]
    Suppose $(R, \m)$ is a Noetherian local ring of residue characteristic $p> 0$ that is a complete intersection. Suppose that $f \in \m$ is a nonzerodivisor and $R/fR$ is \claperfdinj. Then $R$ is \claperfdinj. In fact, we even obtain that the pair $(R, f)$ is \claperfdinj.
\end{theoremB*}

The key point for both these results is that when $R$ is LCI, we show in \autoref{thm: lci case} that $R_{\perfd}$ is Cohen-Macaulay, and then the proof of our inversion of adjunction result follows similarly to classical results in characteristic zero or $p > 0$ \cite{ElkikDeformationsOfRational,FedderFPureRat}.

While we hope that the LCI condition is unnecessary, even this is enough to show that numerous examples are \claperfdpure (for instance $\mathbb{Z}_p[[y,z]]/(p^3 + y^3 + z^3)$ as long as $p \equiv 1 \pmod 3$), see \autoref{ex:CYlikehypersurfaces}.

We can relate our mixed characteristic singularities to those coming from characteristic zero as follows.  Compare with \cite{HaraWatanabeFRegFPure,SchwedeFInjectiveAreDuBois}.

\begin{theoremC*}
    Suppose $(R, \fram)$ is a mixed characteristic $(0, p> 0)$ Noetherian local ring.
    \begin{enumerate}
        \item If $R$ is \absperfdinj and $R$ is essentially of finite type over a DVR, then $R[1/p]$ has Du Bois singularities.  (\autoref{prop.ArcImpliesR1overpIsDuBois})
        \item If $R$ is normal, $\bQ$-Gorenstein, and \claperfdpure, then $R[1/p]$ is log canonical.  (\autoref{prop.PerfdPureImpliesR1overpIsLC}) \label{partb}
        \item If $R$ is normal, $\bQ$-Gorenstein with canonical index not divisible by $p > 0$, and $R$ is \claperfdpure, then $R$ is log canonical.\label{partc} (\autoref{cor.PerfdPureImpliesLC})
    \end{enumerate}
\end{theoremC*}

Indeed, we even obtain \autoref{partc} in the potentially non-normal semi-log canonical case in \autoref{thm.PerfdPurePlusQGorIndexImpliesLC}.  Relatedly, we prove that \absperfdinj singularities are weakly normal in \autoref{cor.ArcInjectiveImpliesWeaklyNormal}.  To prove \autoref{partb} and \autoref{partc}, we study the behavior of perfectoid purity under cyclic covers in \autoref{prop.CyclicCoverNIndexStuff}.

\subsection{Acknowlegdements}
Bhargav Bhatt was supported by NSF Grant DMS \#1801689 and \#1840234, NSF FRG Grant \#1952399, a Packard
Fellowship, and the Simons Foundation.  Linquan Ma was supported in part by NSF Grant DMS \#2302430 and a fellowship from the Sloan Foundation.  Zsolt Patakfalvi was supported by grant \#200020B/192035 from the Swiss National Science Foundation and ERC Starting grant \#804334.  
Karl Schwede was supported by NSF Grants \#2101800, \#2501903 and NSF FRG Grant DMS-1952522, Simons Travel Support SFI-MPSTSM-00013051 and a Fellowship from the Simons Foundation.  Kevin Tucker was partially supported by
NSF Grants \#2200716, \#2501904 and Simons Foundation Travel Support for Mathematicians SFI-MPS-TSM-00014083.  Joe Waldron was supported by NSF CAREER Grant DMS \#2440240, NSF Grant DMS\#2401279, the Simons Foundation Gift \#850684, and also gratefully acknowledges support from the Institute for Advanced Study during his Spring 2024 Membership funded by the Infosys Membership Fund. Jakub Witaszek was supported by NSF Grant No.\ DMS-\#2101897 and NSF Grant No.\ DMS-\#2401360.  This material based upon work supported by the National Science Foundation under Grant No. DMS-1928930, while the authors were in residence at the Simons Laufer Mathematical Sciences Institute (formerly MSRI) in Berkeley, California, during the Spring 2024 semester.  The authors thank Elden Elmanto and George Pappas for valuable conversations.  The authors also thank Anne Fayolle for valuable comments on earlier drafts and in particular for pointing out a problem with {\color{red} 4.29}.  Finally, the authors also thank the referee for numerous useful comments.

\section{Preliminaries}

We begin by recalling the definitions of some singularities in equal characteristic for the convenience of the reader.

Suppose first $R$ is a ring of characteristic $p > 0$.  In order to distinguish between the target and source of the Frobenius endomorphism, we write the Frobenius map as $R \to F_* R$.  The abelian group structure of $F_* R$ is the same as that of $R$, but with the $R$-module structure induced by Frobenius. 

\begin{definition}[$F$-singularities {\cite{HochsterRobertsFrobeniusLocalCohomology,FedderFPureRat}}]
    Suppose $R$ is a Noetherian ring of characteristic $p > 0$.  We say that $R$ is \emph{$F$-pure} if the Frobenius map
    \[ 
        R \to F_* R 
    \]
    is pure\footnote{this is also known as being universally injective}.  We say that $R$ is \emph{$F$-injective} if for each maximal ideal $\fram \subseteq R$, we have that 
    \[
        H^i_{\fram}(R) \to H^i_{\fram}(F_* R)
    \]
    injects for every $i \geq 0$.
\end{definition}

If $R$ is $F$-finite\footnote{meaning Frobenius is a finite map} or if $R$ is complete and local, then $R$ is $F$-pure if and only if $R \to F_* R$ splits as a map of $R$-modules.  It is clear that every $F$-pure ring is $F$-injective.  The converse holds if $R$ is quasi-Gorenstein, that is if $R$ has a canonical module locally isomorphic to $R$, \cite{FedderFPureRat}.  

We also recall some notions of singularities most commonly studied over rings of characteristic zero.  

\begin{definition}[Log canonical singularities]
    Suppose $X$ is a normal Noetherian integral scheme with a dualizing complex and a canonical divisor $K_X$.  We say that $X$ is \emph{log canonical} if $X$ is $\bQ$-Gorenstein\footnote{meaning there is some integer $n > 0$ such that $nK_X$ is Cartier} and for every proper birational map $\pi : Y \to X$ with $Y$ normal, we have that the coefficients of 
    \[
        K_Y - \pi^* K_X
    \]
    are $\geq -1$.  In other words, if $E$ is the reduced exceptional divisor, we require that $\pi_* \cO_Y(\lceil K_Y - \pi^* K_X + \epsilon E\rceil) = \cO_X$ for all $1 \gg \epsilon > 0$. 
\end{definition}
In the presence of a log resolution, one can check whether $X$ is log canonical by simply checking $K_Y - \pi^* K_X$ has coefficients $\geq -1$ on that log resolution $Y$. 

There is also a notion of semi-log canonical singularities which are a non-normal variant of log canonical singularities.  We refer the reader to \cite{KollarKovacsSingularitiesBook} for this definition and for more details on log canonical singularities.  Roughly speaking though, as long as $X$ is semi-log canonical in codimension 1 (has node-like-singularities after localizing at height one primes), the definition is exactly the same as the one above.

We now also recall the definition of Du Bois singularities.

\begin{definition}[Du Bois singularities {\cite{SteenBrinkDuBoisReview,DuBoisMain}}]
    Suppose $X$ is a reduced scheme essentially of finite type over a field of characteristic zero.  We say that $X$ is \emph{Du Bois} if $\cO_X \to \DuBois{X}$ is an isomorphism (here $\DuBois{X}$ is defined as in \cite{DuBoisMain}).
\end{definition}

Every log canonical variety over a field of characteristic zero is Du Bois \cite{KollarKovacsLCImpliesDB}, and a normal quasi-Gorenstein Du Bois variety is automatically log canonical \cite{KovacsDuBoisLC1}.  Since the map 
\[ 
\myR \Hom_{\cO_X}(\DuBois{X}, \omega_X^{\mydot}) \to \omega_X^{\mydot}
\] 
always injects on cohomology, by Grothendieck duality $X$ has Du Bois singularities if and only if the displayed map surjects on cohomology.  Furthermore, by either local duality applied to the above, or the argument of \cite[Lemma 2.2]{KovacsDuBoisLC1} in view of \cite[Theorem 3.3]{KovacsSchwedeDBDeforms}, or by \cite{GodfreyMurayamaPureOfDuBois}, $X$ is Du Bois if and only if 
\[ 
    H^i_{x}(X, {\cO_{X,x}}) \to H^i_{x}(X, {\DuBois{X}}_x)
\]
injects for every point $x \in X$.  Note, this condition is reminiscent of the $F$-injective condition above.  

Via reduction to characteristic $p \gg 0$, (conjecturally) we have that log canonical singularities correspond to $F$-pure singularities, and Du Bois singularities correspond to $F$-injective singularities, \cf \cite{HaraWatanabeFRegFPure,SchwedeFInjectiveAreDuBois,MustataSrinivasOrdinary,BhattSchwedeTakagiweakordinaryconjectureandFsingularity}.

\subsection{Pure maps in $\myD(R)$}
\label{subsec.PureMapsInDR}

In what follows $R$ is a commutative ring.
The notion of purity we consider makes use of colimits in a derived category, and so it is crucial to consider $\myD(R)$ to be the (unbounded) derived $\infty$-category rather than the classical triangulated category which is its homotopy category. 
This can be obtained via an $\infty$-categorical analog of the classical Verdier construction, by performing a Dwyer-Kan localization of $\mathrm{N}(\mathrm{Ch}(R))$ at the quasi-isomorphisms.  For more information about this object, see \cite[Section 1.3.5]{lurie_higher_algebra} and in particular \cite[Definition 1.3.5.8, Proposition 1.3.5.15]{lurie_higher_algebra}. For the reader more familiar with the derived category language, co-fiber in a stable $\infty$-category corresponds to the cone in the corresponding triangulated category (i.e., in its homotopy category), fiber to a cone shifted by $-1$, and a fiber sequence corresponds to an exact triangle.

\begin{definition}
    A map $f : M \to N$ in $\myD(R)$ is \emph{pure} if it can be written as a  filtered colimit of split maps, $f_i : M \to N_i$.
\end{definition}

\begin{remark}
If $M$ and $N$ are discrete $R$-modules then this coincides with the traditional definition, see \cite[\href{https://stacks.math.columbia.edu/tag/058K}{Tag 058K}]{stacks-project}.
\end{remark}

We start with verifying some basic properties:

\begin{lemma}\label{lem:first_factor_is_pure}
Let $f:M\to N$ and $g:N\to L$ be maps in $\myD(R)$, and suppose that  $g\circ f:M\to L$ is pure.  Then $f$ is pure.
    \end{lemma}
\begin{proof}
    Suppose that we can write $g\circ f$ as a colimit of split maps $f_i:M\to L_i$.  Then define $N_i$ as the following pullback.
    \[
    \xymatrix{
       N_i\ar[r]\ar[d]& L_i\ar[d]\\
         N\ar[r] & L
        }
    \]
    By the definition of pullback, we have a canonical map $M\to N_i$ that factors $M\to L_i$, and since the latter splits, we have that $M\to N_i$ splits. It is therefore enough to show that $N=\colim N_i$.  But this follows from taking the colimit of the above diagrams, since filtered colimits in a stable $\infty$-category commute with finite limits.  That latter follows from applying \cite[Proposition 1.1.4.1]{lurie_higher_algebra} to $\mathrm{colim}:\Ind(D(R))\to D(R)$, since colimits always commute with colimits, where $\Ind(D(R))$ is stable by \cite[Proposition 1.1.3.6]{lurie_higher_algebra}.
\end{proof}

\begin{lemma}\label{lem:composition_is_pure}
    If $f:M\to N$ and $g:N\to L$ are pure maps in $\myD(R)$ then $g\circ f: M\to L$ is pure.
\end{lemma}
\begin{proof}
    Suppose that we can write $g$ as a colimit of split maps $g_i:N\to L_i$. 
 Then it suffices to show that ${ g_i \circ f} : M\to L_i$ is pure since it follows from the definition that the colimit of pure maps is pure.  Since $N\to L_i$ is split, we obtain a factorization $M\to N\to L_i\to N$, 
    and therefore $M\to L_i$ is pure by \autoref{lem:first_factor_is_pure} as required.
\end{proof}

\begin{lemma}\label{lem:tensor-product-food}
If $\phi:M\to N$ is a pure map in $\myD(R)$ and $K\in \myD(R)$ then $M\otimes_R^\myL K\to N\otimes_R^{\myL} K$ is pure. 
\end{lemma}
\begin{proof}
Since $\phi: M\to N$ is pure, there exists a system of split maps $\phi_i:M\to N_i$ with $\phi=\colim \phi_i$.  Then $M\otimes_R^{\myL}K\to N_i\otimes_R^{\myL} K$ is split.  Hence $$M\otimes_R^{\myL} K\to \colim(N_i\otimes_R^{\myL} K)\simeq (\colim N_i)\otimes_R^{\myL} K=N\otimes_R^\myL K$$ is pure as required. 
\end{proof}

The following is well known for maps of modules, we expect our generalization to complexes below is also well known to experts. 
\begin{lemma}\label{lem:injective_E_purity}
    Suppose $(R, \fram)$ is a Noetherian local ring and denote the injective hull of the residue field by $E$.  Suppose $M \in \myD(R)$, and there is a map $f : R \to M$.  Suppose that 
    \[
        R \otimes_R E  \to  H^0(M \otimes_R^{\myL} E)
    \]
    is injective. Then there is a map $M \to R^{\wedge\m}$ such that the composition 
    \[
        R \to M \to R^{\wedge\m}
    \]
    is the $\fram$-adic completion map.  In particular, $R\to M$ is pure, and if $R$ is complete local then $R \to M$ splits.
\end{lemma}
\begin{proof}
    Since $E$ is injective, the exact functor $\Hom(-, E) = \myR\Hom(-, E)$ turns injections into surjections.  This yields a map
    \[
        \myR \Hom(R \otimes_R^{\myL} E, E) \leftarrow  \myR \Hom(M \otimes_R^{\myL} E, E)
    \] 
    which we identify with 
    \[
        \myR \Hom(R, \myR \Hom(E, E)) \leftarrow \myR \Hom(M, \myR \Hom( E, E)).
    \]
    This map is surjective on cohomology.
    Notice that $\myR\Hom(E, E) = R^{\wedge\m}$ (the identity maps to $1$).  If we take zeroth cohomology we get a surjective map 
    \[
        \Hom_{\myD(R)}(R, R^{\wedge\m}) \leftarrow \Hom_{\myD(R)}(M, R^{\wedge\m}).
    \]
    This implies there is a map $M \to R^{\wedge\m}$ splitting $f : R \to M$ in the sense of the statement, as desired.  Since $R^{\wedge\m}$ is a faithfully flat $R$-module, $R\to R^{\wedge\m}$ is pure and so $R\to M$ is pure by \autoref{lem:first_factor_is_pure}.
\end{proof}

\begin{proposition}\label{prop:purity-criterion}
Let $f:M\to N$ be a map in $\myD(R)$ with cone $Q$.  Then $f$ is pure if and only if for every perfect complex $K$ with a map $K\to Q$, the composition $K\to M[1]$ is zero. 
\end{proposition}
\begin{proof}
Suppose that $M\to N$ is a {filtered} colimit of split maps $f_i:M\to N_i$.  Then let $Q_i$ be the cone of $M\to N_i$.  Since perfect complexes are the compact objects of $D(R)$ by \cite[Tag 07LT]{stacks-project}, and $Q_i$ and $K$ are perfect,  
we have a factorization $K\to Q_i\to Q$ for some $i$.  
Then the required vanishing follows since $M\to N_i\to Q_i$ splits, and $K\to M[1]$ factors through $K\to Q_i\xrightarrow{0} M[1]$.

Conversely suppose that the property holds.  We may express $Q=\colim K_i$ as a filtered colimit of perfect complexes since perfect complexes are the compact objects of $D(R)$ and $D(R)$ is compactly generated by \cite[Remark 1.4.4.3]{lurie_higher_algebra} and \cite[Proposition 2.5]{neeman_grothendieck_duality}.

Define $N_i$ via the pullback
 \[
    \xymatrix{
        M \ar[r]\ar^@{=}[d] & N_i\ar[r]\ar[d]& K_i\ar[d]\\
        M\ar[r] & N\ar[r] & Q
        }
    \]
Then we have $N\simeq \colim N_i$  because filtered colimits commute with finite limits.  Furthermore, since $K_i\to M[1]$ is the zero map, the top row splits {by \cite[1.2.7]{neeman_triangulated}}.  Note that in this context, the colimit of zero maps need not be zero, and hence we cannot conclude that $N\to Q$ is itself split.
\end{proof}

Next we recall that if $I=(f_1,\dots,f_n)$ is a finitely generated ideal of a commutative ring $R$ and $M\in D(R)$, then following \cite[\href{https://stacks.math.columbia.edu/tag/0952}{Tag 0952}]{stacks-project}, one defines $\myR\Gamma_I M := C^\bullet(\underline{f},R)\otimes_RM$, where $C^\bullet(\underline{f},R)$ denotes the \v{C}ech complex associated to the sequence $f_1,\dots,f_n$. It is well-known that this definition only depends on $I$ \cite[\href{https://stacks.math.columbia.edu/tag/0A6R}{Tag 0A6R}]{stacks-project} and that when $R$ is Noetherian and $M$ is a finitely generated $R$-module, it agrees with the usual definition of local cohomology using injective resolutions, see \cite[\href{https://stacks.math.columbia.edu/tag/0BJD}{Tag 0BJD}]{stacks-project}.

\begin{proposition}\label{prop:purity_implies_injecitivity}
Let $R$ be a Noetherian ring and $I\subseteq R$ an ideal, and let $f:M\to N$ be a pure map in $\myD(R)$. Then $H^i\myR\Gamma_{I}M\to H^i\myR\Gamma_{I}N$ is injective for all $i$. 
\end{proposition}
\begin{proof}
Write $f=\colim f_i$ where $f_i:M\to N_i$ splits. 
 Since $H^i\myR\Gamma_I M\to H^i\myR\Gamma_I N_i$ is injective for each $i$, we have 
 \[H^i\myR\Gamma_I M\hookrightarrow \colim H^i\myR\Gamma_I N_i\cong  H^i\myR\Gamma_I \colim N_i\cong H^i\myR\Gamma_I N. \qedhere \]
\end{proof}

\begin{definition}\label{defn:p-completely-pure}
Let $R$ be an $I$-complete Noetherian ring and $M,N\in\widehat{\myD}(R)$, the derived $I$-complete objects in $\myD(R)$.  A map $M\to N$ is \emph{$I$-completely pure} if $$M\otimes_R^{\myL}R/I^n\to N\otimes_R^{\myL}R/I^n$$ is pure for all $n$. 
\end{definition}

\begin{lemma}\label{lem:p-completely-pure-implies-pure}
Let $R$ be an $I$-complete Noetherian ring, $M, N\in\widehat{\myD}(R)$ such that $M$ is bounded above with coherent cohomology (equivalently, it is pseudo-coherent \cite[\href{https://stacks.math.columbia.edu/tag/064Q}{Tag 064Q}]{stacks-project}) and
$M\to N$ an $I$-completely pure map.  Then $M\to N$ is pure.
\end{lemma}
\begin{proof}
Let $Q$ be the cone of $M\to N$.  By \autoref{prop:purity-criterion} we must show that for any perfect $K\in D(R)$ with map $K\to Q$, the composition $K\to Q\to M[1]$ is $0$.  
Fixing such a $K \to Q \to M[1]$ gives an element $x\in H^1(L)$, where $L= \myR\Hom_R(K, M)$, which we need to show is zero.  
By \cite[\href{https://stacks.math.columbia.edu/tag/0EGV}{Tag 0EGV}]{stacks-project}, we have $H^i(L)\simeq \lim_n H^i(L\otimes_R^\myL R/I^n)$, and so it is enough to show that the image of $x$ in $H^1(L\otimes_R^\myL R/I^n)$ is zero.  
By \cite[\href{https://stacks.math.columbia.edu/tag/0A6A}{Tag 0A6A}]{stacks-project} we have that $$L\otimes_R^\myL R/I^n\cong \myR \Hom_{R/I^n}(K\otimes_R^\myL R/I^n,M\otimes_R^\myL R/I^n).$$
 Since $M\to N$ is $I$-completely pure, the  map $K\otimes^\myL R/I^n\to M\otimes^\myL R/I^n[1]$ which corresponds to the image of $x$ in $H^1(L\otimes_R^\myL R/I^n)$ via the above equivalence
is zero, which shows that the image of $x$ in $ H^1(L\otimes_R^\myL R/I^n)$ is zero as required.
\end{proof}

\section{Absolute perfectoidization}

In positive characteristic, we can measure singularities of an excellent local ring by comparing it to its perfection.  
If  $(A,I)$ is a perfect prism corresponding to perfectoid ring $\overline{A}:=A/I$, and $R$ is a $p$-complete
 $\overline{A}$-algebra, the perfectoidization of $R$ is defined in \cite[Section 8]{BhattScholzepPrismaticCohomology} as
\[
    R_{\perfd}:=(\colim_{e} \phi^e_*\Prism_{R/A})^{\wedge{(p, I)}} \otimes^{\myL}_A\overline{A}\in D(\overline{A}).
\]
However, we are primarily interested in Noetherian rings, for which there will be no such choice of perfectoid base.
The aim of this section is to introduce a generalization of the above construction to the Noetherian situation using the framework of \cite{BhattLurieAbsolute} and prove its basic properties.

\subsection{The Cartier-Witt stack}

In this subsection we briefly introduce the Cartier-Witt stack, and describe the main properties necessary for the construction of $R_{\perfd}$.  For further information see \cite[Section 3]{BhattLurieAbsolute}.

\begin{definition}
    A \emph{generalized Cartier divisor} of a scheme $X$ is a pair $(\sI,\alpha)$ where $\sI$ is an invertible sheaf and $\alpha:\sI\to\cO_X$ is a morphism of $\cO_X$-modules.  A morphism between such objects is defined to be an isomorphism $\rho:\sI\to\sI'$ satisfying $\alpha=\alpha'\circ\rho$.
    Denote the category of generalized Cartier divisors of $X$ by $\Cart(X)$.
    $\Cart$ forms a stack for the fpqc topology, which can be identified with $[\bA^1/\bG_m]$.
\end{definition}

\begin{definition}\cite[Definition 3.1.4]{BhattLurieAbsolute}
    Let $R$ be a commutative ring in which $p$ is nilpotent, and $W(R)$ be the ring of Witt vectors of $R$.  We say a generalized Cartier divisor $(I,\alpha)$ of $\Spec(W(R))$ is a \emph{Cartier-Witt divisor} of $R$ if
    \begin{enumerate}
        \item The image of $I\xrightarrow{\alpha}W(R)\to R$ is a nilpotent ideal.
        \item The image of $I\xrightarrow{\alpha} W(R)\xrightarrow{\delta} W(R)$ generates the unit ideal. 
    \end{enumerate}
    Denote $\WCart(R)$ to be the full subcategory of $\Cart(W(R))$ spanned by the Cartier-Witt divisors.  This functor gives rise to the \emph{Cartier-Witt stack}, which is a stack for the fpqc topology. If $p$ is not nilpotent in $R$ we set $\WCart(R)=\emptyset$.
\end{definition}

\begin{remark}\cite[Construction 3.2.4]{BhattLurieAbsolute}
Given a prism $(A,I)$, we obtain a morphism of stacks $\rho_A:\Spf(A)\to \WCart$, where $\Spf$ is taken in the $(p,I)$-adic topology on $A$.  
To produce this functor, it is enough to associate a Cartier-Witt divisor to every homomorphism $f:A\to R$ for which $(p,I)$ is nilpotent in the image.  By the universal property of $W(R)$ there is a unique lift $\tilde{f}:A\to W(R)$ as a map of $\delta$-rings, and then $I\otimes_A W(R)\to W(R)$ provides the required Cartier-Witt divisor.
\end{remark}

\begin{remark}\label{rem:WCart_Frobenius}
    Let $R$ be a commutative ring.  Then pullback by the Frobenius $W(R)\to W(R)$  defines a functor $\phi:\WCart(R)\to \WCart(R)$, which in turn gives a Frobenius on the Cartier-Witt stack $\phi:\WCart\to \WCart$.  For any prism $(A,I)$ this gives a diagram
    \[
\xymatrix{
\Spf(A)\ar^{\rho_A}[r]\ar^{\phi}[d] & \WCart\ar^{\phi}[d]\\
\Spf(A)\ar^{\rho_A}[r] &\WCart
}
\]
which commutes up to canonical isomorphism.
\end{remark}

\begin{proposition}\cite[Definition 3.3.1/Proposition 3.3.5]{BhattLurieAbsolute}\label{prop:D_WCart_description}
For a ring $R$, let $\myD(R)$ denote the derived $\infty$-category of $R$-modules.  Then the \emph{$\infty$-category of quasi-coherent complexes on $\WCart$} is 
\[\myD(\WCart):=\varprojlim_{\Spec(R)\to\WCart}\myD(R)\simeq \varprojlim_{(A,I)}\widehat{\myD}(A)\]
where the first limit is indexed over all points of $\WCart$, while the second is indexed by the category of all bounded prisms, and $\widehat{\myD}(A)$ indicates the full subcategory of $\myD(A)$ spanned by $(p,I)$-complete complexes. 
\end{proposition}

\begin{definition}\cite[Construction 4.4.1]{BhattLurieAbsolute}
Let $R$ be a commutative ring.  For a bounded prism $(A,I)$, let $\Prism_{\bullet/A}\in\widehat{\myD}(A)$ denote the relative prismatic complex.  For any morphism of bounded prisms $(A,I)\to (B,IB)$, there is a canonical isomorphism \[B\widehat{\otimes}^{\myL}_A\Prism_{(\overline{A}\otimes_{\mathbf{Z}_p}^{\myL} R)/A}\to \Prism_{(\overline{B}\otimes_{\mathbf{Z}_p}^{\myL} R)/B}.\]
It follows from \autoref{prop:D_WCart_description} that this determines an object $\HPrism(R)\in \myD(\WCart)$ called the \emph{prismatic cohomology sheaf} of $R$.
\end{definition}

\begin{definition}\cite[Definition 3.4.1]{BhattLurieAbsolute}
The \emph{Hodge-Tate divisor} $\WCart^{\HT}$ is the closed substack of $\WCart$ given by 
\[\WCart^{\HT}(R)=\{(I,\alpha)\in\WCart(R) \mid \textrm{\ the\ composition\ } I\xrightarrow{\alpha} W(R)\to R \textrm{\ is\ }0\}.\]
\end{definition}

\begin{remark}
    For any prism $(A,I)$, \cite[Remark 3.4.2]{BhattLurieAbsolute} gives a pullback square 
\[
\xymatrix{
\Spf(A/I)\ar^{\rho_A^{\HT}}[r]\ar[d] & \WCart^\HT\ar[d]\\
\Spf(A)\ar^{\rho_A}[r] &\WCart
}
\]
Hence for any perfectoid ring $\overline{A}$, which is the quotient of a unique (up to unique isomorphism) perfect prism $(A,I)$, we get a functorial morphism $\rho_A^{\HT}:\Spf(\overline{A})\to \WCart^{\HT}$.
\end{remark}

The perfectoidization of a ring will be a sheaf on the Hodge-Tate divisor, so finally we record the analog of \autoref{prop:D_WCart_description}.

\begin{proposition}\cite[Definition 3.5.1/Remark 3.5.3]{BhattLurieAbsolute}
\label{prop:DescriptionOfDWCartHT}
    The $\infty$-category of quasi-coherent complexes on the Hodge-Tate divisor is 
    \[\myD(\WCart^{\HT}):=\varprojlim_{\Spec(R)\to \WCart^{\HT}}\myD(R)\simeq \varprojlim_{(A,I)} \widehat{\myD}(A/I)\]
    where the first limit runs over all points of the Hodge-Tate divisor while the second runs over all bounded prisms.
\end{proposition}

\subsection{Perfection}  

With the background out of the way, we are ready to introduce $R_{\perfd}$ and prove its main properties. 

\begin{definition}\label{def:perfd}
For a commutative ring $R$, we define the perfectoidization $R_{\perfd}$ by

\[\HPrism(R)_{\perf}:=\varinjlim (\HPrism(R)\to \phi_*\HPrism(R)\to \phi_*^2\HPrism(R)\to\cdots)\in\myD(\WCart)\]

\[R_{\perfd}:=\myR\Gamma(\WCart^{\HT},\HPrism(R)_{\perf}|_{\WCart^{\HT}})\in\widehat{\myD}(\mathbf{Z}_p)\]
\end{definition}

\begin{remark}
Since colimit commutes with $\myR\Gamma(\WCart^{\HT},-)$ \cite[Corollary 3.5.13]{BhattLurieAbsolute} and restriction, we have 
$$R_{\perfd}=\colim_e \myR\Gamma(\WCart^{\HT},\phi_*^e\HPrism(R)|_{\WCart^{\HT}})$$
Note that every term in this colimit is in $\widehat{D}(R)$ and we take the colimit in $\widehat{D}(R)$. 
\end{remark}

\begin{lemma}
    If $R$ is an $\overline{A}$-algebra for some perfectoid ring $\overline{A}$, then $R_{\perfd}$ agrees with the perfectoidization defined in \cite{BhattScholzepPrismaticCohomology}.
\end{lemma}
\begin{proof}
Everything that follows occurs in $\widehat{\myD}({\bZ}_p)$, so there are $p$-completions built in.  We omit them from the notation.  The result follows from the following chain of equivalences:
\begin{align*}
R_{\perfd}&:=\myR\Gamma\left(\WCart^{\HT},\left(\colim\left(\phi_*^e\HPrism(R)\right)\right)|_{\WCart^{\HT}}\right)&\\
    &\simeq \myR\Gamma\left(\WCart^{\HT},\colim\left(\phi_*^e\HPrism(R)|_{\WCart^{\HT}}\right)\right)& \text{colim commutes with restriction}\\
    &\simeq \colim \myR\Gamma(\WCart^{\HT},\phi_*^e\HPrism(R)|_{\WCart^{\HT}}) &\text{\cite[3.5.13]{BhattLurieAbsolute}}\\
    &\simeq \colim \myR\Gamma(\WCart^{\HT},\phi_*^e\rho_{A*}\Prism_{R/A}|_{\WCart^{\HT}})& \text{\cite[Proposition 4.4.8]{BhattLurieAbsolute}}\\
    &\simeq\colim \myR\Gamma(\WCart^{\HT},\rho_{A*}\phi_*^e\Prism_{R/A}|_{\WCart^{\HT}})& \text{\autoref{rem:WCart_Frobenius}}\\
    &\simeq\colim \myR\Gamma\left(\Spf(\overline{A}),
    (\phi_*^e\Prism_{R/A}\otimes \overline{A}) \right)&\\
    &\simeq \myR\Gamma(\Spf(\overline{A}),(\colim \phi_*^e\Prism_{R/A})\otimes\overline{A}). 
    \end{align*} 
\end{proof}

\begin{remark}
By the lemma above, we know that $R_{\perfd}$ is a perfectoid ring when $R$ is semi-perfectoid \cite[Theorem 7.4]{BhattScholzepPrismaticCohomology} (see also \cite{Ishizuka} for an explicit characterization of $R_{\perfd}$ in the semi-perfectoid case). In general, $R_{\perfd}$ has the structure of a derived commutative ring (see e.g. \cite[Section 4]{RaksitHochschildHomology}), but we do not need this {in our paper}. 
\end{remark}

\begin{proposition}\label{prop:universal_property_perfd}
    For a commutative ring $R$ with $p$ in its Jacobson radical, we have
    \[R_{\perfd}=\varprojlim_{R\to B} B\] where the limit runs over all maps from $R$ to perfectoid rings $B$. 
    This limit is computed in $\widehat{\myD}(\mathbf{Z}_p)$.
\end{proposition}
\begin{proof}
By \cite[Proposition 8.5]{BhattScholzepPrismaticCohomology}, 
the proposition holds in the case where $R$ is an $\overline{A}$-algebra for some perfectoid ring $\overline{A}$. It follows that $(\overline{A}\widehat{\otimes}^{\myL}_{\mathbf{Z}_p}R)_{\perfd} \cong \varprojlim_{R\to B} B$.

Now for a general prism $(A,I)$, we have 
\[\HPrism(R)_{\perf}|_{\WCart^{\HT}}(\overline{A})= (\colim_e \phi_*^e\Prism_{\overline{A}\widehat{\otimes}^{\myL}_{\mathbf{Z}_p} R/A})^{\wedge(p, I)}\widehat{\otimes}_A \overline{A}\]
Let $A_{\perf}$ be the perfection of $A$, i.e., $A_{\perf}=(\colim_e\phi_*^eA)^{\wedge(p,I)}$. Then $A_\perf$ is a perfect prism and we have canonical maps 
$$A_\perf \to (\colim_e \phi_*^e\Prism_{\overline{A}\widehat{\otimes}^{\myL}_{\mathbf{Z}_p} R/A})^{\wedge(p,I)}\xrightarrow{\alpha} (\colim_e\phi_*^e\Prism_{\overline{A}_\perf \widehat{\otimes}^{\myL}_{\mathbf{Z}_p} R / A_\perf})^{\wedge(p,I)}.$$
Note that by base change,
$$(\colim_e\phi_*^e\Prism_{\overline{A}_\perf \widehat{\otimes}^{\myL}_{\mathbf{Z}_p} R / A_\perf})^{\wedge(p,I)} \cong (\colim_e \phi_*^e\Prism_{\overline{A}\widehat{\otimes}^{\myL}_{\mathbf{Z}_p} R/A})^{\wedge(p,I)} \widehat{\otimes}^{\myL}_AA_\perf.$$
Thus $\alpha$ admits a section, namely the multiplication map $\mu$ (in fact, one can show that $\alpha$ is an isomorphism). Therefore, when computing $\HPrism(R)_{\perf}|_{\WCart^{\HT}}$, we may restrict ourselves to perfect prisms. By \autoref{prop:DescriptionOfDWCartHT} and the discussion above, we have
$$R_{\perfd}= \varprojlim_{(A,I)}\left(\varprojlim_{\overline{A}\widehat{\otimes}^{\myL}_{\mathbf{Z}_p}R \to B} B\right)$$
where the first limit runs over all perfect prisms $(A,I)$. By \cite[Theorem 3.10]{BhattScholzepPrismaticCohomology}, we can rewrite the above as 
$$R_{\perfd}= \varprojlim_{S}\left(\varprojlim_{S\to B, R \to B} B\right)$$
where the first limit runs over all perfectoid rings $S$ and the second limit runs over all maps of perfectoid rings $S\to B$ and all maps of rings $R\to B$. Finally, we note that the functor $\Phi$ from the category 
$$\{S\to B \text{ map of perfectoid rings}, R\to B \text{ map of rings}\}$$
to the category of perfectoid $R$-algebras sending $\{S\to B, R\to B\}$ to $B$ is left adjoint to the functor sending $B$ to $\{B\xrightarrow{=}B, R\to B\}$. Therefore, pulling back diagrams along $\Phi$ does not change the limit, that is, 
$$R_{\perfd}= \varprojlim_{R\to B}B$$
where the limit runs over all $R\to B$ where $B$ is perfectoid. 
\end{proof}

\begin{proposition}\label{prop:h-sheafification}
For a commutative ring $R$ with $p$ in its Jacobson radical, 
we have
    \[R_{\perfd}=\myR\Gamma_{\arc}(\Spf(R),\cO).\]
In particular, $R_{\perfd} \in D^{\geq0}(R)$.
\end{proposition}
\begin{proof}
    The proof from \cite[Proposition 8.10, Corollary 8.11]{BhattScholzepPrismaticCohomology} did not make use of the perfectoid base, so already applies in our setting using \autoref{prop:universal_property_perfd}. 
\end{proof}

\begin{remark}
One could define a version of the $h$-topology for formal schemes which agrees with the arc topology in our Noetherian situation, however we do not pursue this since we do not need it.  In particular $R_{\perfd}$ could be computed in terms of an $h$-sheafification for Noetherian $R$.
\end{remark}

We will need the following result from forthcoming work  \cite{BhattLuriepadicRHmodp}. 

\begin{proposition}
\label{prop:galois_pullback}
Let $\cO_C/\mathbf{Z}_p$ be the $p$-completed ring of integers in a perfectoid extension $C/\mathbf{Q}_p$ which is the $p$-completion of a totally ramified Galois extension of $\mathbf{Q}_p$ with Galois group $\Gamma=\mathbf{Z}_p$. Write $\gamma \in \mathbf{Z}_p$ for a generator. (One can produce such an extension from, e.g., the cyclotomic extension.)

Then for any 
ring $R$, we have a pullback square in $\widehat{D}(R)$:
\[
\xymatrix{ 
R_{\perfd} \ar[r] \ar[d] &  \mathrm{fib}\left( (R \widehat{\otimes}^{\myL}_{\mathbf{Z}_p} \cO_C)_{\perfd} \xrightarrow{\gamma-1}  (R \widehat{\otimes}^{\myL}_{\mathbf{Z}_p} \cO_C)_{\perfd} \right) \ar[d] \\
(R/p)_\perf \ar[r] &  (R/p)_\perf \oplus (R/p)_\perf[-1],
}
\]
where the lower horizontal map is the evident one into the first summand on the right,  the top horizontal map is induced by the natural map to the first term of the right hand side, the left vertical one is the natural one, while the right vertical one is induced by observing that $\gamma$ acts trivially on $(R/p)_\perf$ (so the bottom right entry is also $\mathrm{fib}((R/p)_\perf \xrightarrow{\gamma-1} (R/p)_\perf)$).
\end{proposition}
\begin{proof}[Sketch of Proof]
We apply \cite[Proposition 4.5.1]{BhattLectureNotesRiemannHilbert} to $\mathfrak{X}=\Spf(\cO_C)$ and $\mathfrak{Y}=\Spf(\mathbf{Z}_p)$ (see \cite[Example 4.5.2]{BhattLectureNotesRiemannHilbert} for $K=\mathbf{Q}_p$). Then for any $\mathbf{Z}_p$-algebra $R$, consider the sheaf $R_{\perfd}$ on $\Spf(\mathbf{Z}_p)^{\mathrm{pfd}}$, the fiber square in \cite[Example 4.5.2]{BhattLectureNotesRiemannHilbert} yields the pullback square in the proposition.
\end{proof}

\begin{proposition}\label{prop:formally_etale_perfd}
Let $R\to S$ be a map of rings such that $p$ is in their Jacobson radicals.
Suppose $R/p\to S/p$ is relatively perfect, i.e., the relative Frobenius for the animated $\mathbf{F}_p$-algebra map $R \otimes^{\myL}_{\mathbf{Z}} \mathbf{F}_p \to S \otimes^{\myL}_{\mathbf{Z}} \mathbf{F}_p$ is an isomorphism. Then we have $S\widehat{\otimes}^{\myL}_R R_{\perfd}\simeq S_{\perfd}$.
\end{proposition}
\begin{proof}
We may assume that both $R$ and $S$ are $p$-complete. It suffices to check the equivalent statement for the other three corners of the pullback square appearing in \autoref{prop:galois_pullback}.  In the top right, it is sufficient to show that $(R\widehat{\otimes}^{\myL}_{\mathbf{Z}_p}\cO_C)_{\perfd}\widehat{\otimes}_R^{\myL}S\cong (S\widehat{\otimes}^{\myL}_{\mathbf{Z}_p}\cO_C)_{\perfd}$, and so assume we are working over the perfectoid base $\cO_C$. Furthermore, the rings on the bottom row are $\mathbb{F}_p$ algebras.  Therefore it suffices to show that the proposition holds for when $R$ is an algebra over some perfectoid ring $\overline{A}$, corresponding to the perfect prism $(A,d)$.

In this case we have $R_{\perfd}=(\colim_n \phi_*^n\Prism_{R/A})^{\wedge(p, I)} \otimes^{\myL}_A\overline{A}$. 
We can rewrite 
$$R_{\perfd}=
\big(\colim (\phi_*^n\Prism_{R/A}\widehat{\otimes}^{\myL}_A\overline{A})\big)^{\wedge p}.$$  
Denote $R_{\HT,n}=\phi_*^n{\Prism}_{R/A}\widehat{\otimes}^{\myL}_A \overline{A}$.  Therefore it suffices to show that $R_{\HT,n}\widehat{\otimes}^{\myL}_R S\simeq {S}_{\HT,n}$.
For $n=0$, the Hodge-Tate comparison provides a filtered homomorphism $\overline{\Prism}_{R/A}\widehat{\otimes}^{\myL}_R S\to\overline{\Prism}_{S/A}$ whose 
 $i$\textsuperscript{th} graded piece is 
\[\wedge^i L_{R/(A/I)}\{-i\}[-i]^{\wedge p}\widehat{\otimes}^{\myL}_RS\to \wedge^i L_{S/(A/I)}\{-i\}[-i]^{\wedge p}.\]
But since $L_{S/R}^{\wedge p}\simeq 0$,
the above is an equivalence, and hence so is $R_{\HT,0}\otimes_RS\to S_{\HT,0}$.

Now for higher $n$, it is sufficient to check modulo $\phi^{-n}(d)$ by derived Nakayama.  We have

\[(\phi_*^n\Prism_{R/A}/d\widehat{\otimes}^{\myL}_R S)/\phi^{-n}(d)\to \phi_*^n\Prism_{S/A}/(d,\phi^{-n}(d))\]
which identifies with
\[F_*^n(\Prism_{R/A}/(\phi^{n}(d),d))\widehat{\otimes}^{\myL}_R S\to F_*^n(\Prism_{S/A}/(\phi^n(d),d))\]
where we used that $p\in(\phi^n(d),d)$ and hence the modules involved are also $R/p$-modules. Using the factorization $S/p\to S/p\otimes^{\myL}_{R/p} F_*(R/p)\to F_*(S/p)$ and the fact that the latter is an isomorphism, 
this is equivalent to
\[F_*^n(\Prism_{R/A}/(\phi^n(d),d)\widehat{\otimes}^{\myL}_RS )\to F_*^n(\Prism_{S/A}/(\phi^n(d),d))\]
which is an equivalence by the case $n=0$. 
\end{proof}

\begin{remark}
\label{rmk:formally_etale_perfd}
A similar argument actually proves a slight strengthening: If $R\to S$ is a map of rings such that $R/J\to S/J$ is relatively perfect for some finitely generated ideal $J$ that contains $p$. Then we have $(S/J)\widehat{\otimes}^{\myL}_R R_{\perfd}\simeq S_{\perfd}/J$ (here $R/J$, $S/J$ and $S_\perfd/J$ should be interpreted in the derived sense, i.e., if $J=(f_1,\dots,f_n)$, then $R/J\cong \text{Kos}(f_1,\dots,f_n, R)$).
\end{remark}

\begin{proposition}\label{prop:perfds_completion_maximal_ideal}
Let $(R,\fram)$ be a Noetherian local ring with $p\in\fram$.  Then we have
\[(R^{\wedge\m}_{\perfd})^{\wedge\m}\simeq (R_\perfd)^{\wedge\m}\]
  where the (derived) completions are with respect to $\fram$.    
\end{proposition}

\begin{proof}
By derived Nakayama, it suffices to show that $R_{\perfd}/\m \cong R^{\wedge\m}_\perfd/\m$ (again, both sides are interpreted in the derived sense). But this follows from \autoref{rmk:formally_etale_perfd} since $\text{Kos}(\m, R)\to \text{Kos}(\m, R^{\wedge \m})$ is relatively perfect (it is already an isomorphism). 
\end{proof}

\begin{lemma}\label{lem:perfd_localization}
Let $R$ be a Noetherian ring with $p$ in its Jacobson radical. Suppose $Q$ is a prime ideal with $p\in Q$.  
Then we have \[(R_Q)_{\perfd}\simeq (R_{\perfd})_Q^{\wedge p}.\]
\end{lemma}
\begin{proof}
Since $R/p\to R_Q/p=(R/p)_Q$ is relatively perfect, \autoref{prop:formally_etale_perfd} shows that $R_{\perfd}\widehat{\otimes}_R^{\myL}R_Q\simeq (R_Q)_{\perfd}$, which is exactly what we want to show.
\end{proof}

\section{Definitions and basic properties of singularities}

We start by defining the singularities that we will study. 

\begin{definition}
Let $R$ be a Noetherian ring with $p$ in its Jacobson radical. We say $R$ is
\begin{enumerate}
\item \emph{\claperfdpure} if there exists a perfectoid $R$-algebra $B$ such that $R\to B$ is pure;
\item \emph{\absperfdpure} if $R\to R_\perfd$ is pure in $\myD(R)$;
\item \emph{\claperfdinj} if there exists a perfectoid $R$-algebra $B$ such that for every maximal ideal $\fram$ and every $i$, $H^i_\mathfrak{m}(R)\to H^i_{\mathfrak{m}}(B)$ is injective;
\item \emph{\absperfdinj} if for every maximal ideal $\fram$ and every $i$, $H^i_\mathfrak{m}(R)\to H^i_{\mathfrak{m}}(R_{\perfd})$	is injective.
\end{enumerate}
\end{definition}

\begin{remark}\label{rem:atlas_implies_absolute}
Note that \claperfdpure (resp.\ \claperfdinj) implies \absperfdpure (resp.\ \absperfdinj), because if $B$ is the perfectoid $R$-algebra verifying one of the former properties, we have a factorization $R\to R_{\perfd}\to B$ by \autoref{prop:universal_property_perfd}.  In the case of injectivity, this is a standard property, while for purity it follows from  \autoref{lem:first_factor_is_pure}.
\end{remark}

\begin{remark}
\label{rmk:PerfdPureInjectiveCharp}
Suppose $R$ is a Noetherian ring of characteristic $p$. Then $R$ is \claperfdpure  (respectively \claperfdinj) if and only if $R$ is \absperfdpure (respectively \absperfdinj) if and only if $R$ is $F$-pure (respectively $F$-injective). 
This is because in characteristic $p$, perfectoid rings are exactly perfect rings, and $R_{\perfd}=R_{\perf}$ is a discrete ring.
$R\to R_{\perf}$ is pure (respectively induces an injection on local cohomology supported at each maximal ideal) if and only if $R$ is $F$-pure (respectively $F$-injective) essentially by definition. 
\end{remark}

\begin{lemma}\label{lem:pure_injective_comparison}
Let $R$ be a Noetherian ring with $p$ in its Jacobson radical.  If $R$ is \claperfdpure (resp. \absperfdpure), then it is \claperfdinj 
 (resp. \absperfdinj).  The converse holds if $R$ is quasi-Gorenstein (resp.Gorenstein).
\end{lemma}
\begin{proof}

The first assertion follows from \autoref{prop:purity_implies_injecitivity}. We now assume that $R$ is quasi-Gorenstein and perfectoid injective. Our goal is to show that $R$ is perfectoid pure.   Since $R$ is quasi-Gorenstein, for every maximal ideal $\m$ of $R$ with $d=\dim(R_\m)$, we have $H_\m^d(R)\cong E(R/\m)$ and $H_\m^d(B)\cong B\otimes_R E(R/\m)$ for any $R$-algebra $B$, where $E(R/\m)$ denotes the injective hull of $R/\m$. Thus if $R$ is quasi-Gorenstein and \claperfdinj, then there exists some $R$-algebra $B$ such that $E(R/\m)\to B\otimes_R E(R/\m)$ is injective for every maximal ideal $\m$. This implies $R\to B$ is pure and thus $R$ is \claperfdpure by \autoref{lem:injective_E_purity}.

Suppose $R$ is \absperfdinj and Gorenstein.  Then we have an injection
\[
    E = H^d_{\fram}(R) \to H^d_{\fram}(R_{\perfd}) = \myH^d \big( \myR\Gamma_{\fram}(R) \otimes_R^{\myL} R_{\perfd} \big)=\myH^d\big(E[-d]\otimes_R^{\myL}R_{\perfd}\big)
\]
 so by \autoref{lem:injective_E_purity}, $R\to R_{\perfd}$ is pure.
\end{proof}

\begin{lemma}
\label{rmk:make B aic}
Let $R$ be a Noetherian ring with $p$ in its Jacobson radical. Suppose $R$ is \claperfdpure { (respectively \claperfdinj)}. Then we can always choose $B$ perfectoid such that $R\to B$ is pure {(respectively induces an injection on local cohomology supported at each maximal ideal of $R$)}, and such that every element of $R$ has a compatible system of $p$-power roots in $B$ (in fact, we can even assume that $B$ is absolutely integrally closed). 
\end{lemma}
\begin{proof}
By an iterated use of Andr\'{e}'s flatness lemma \cite[Theorem 7.14]{BhattScholzepPrismaticCohomology} we can construct a $p$-complete faithfully flat extension $B\to B'$ of perfectoid rings such that all elements of $R$ have compatible system of $p$-power roots in $B'$ (in fact, we can assume $B'$ is absolutely integrally closed). Note that this implies $B/p^n\to B'/p^n$ is faithfully flat for all $n$.

{In the \claperfdpure case}, it follows that $R/p^n\to B'/p^n$ is pure for every $n$. Now for every maximal ideal $\m$ of $R$, let $E(R/\m)$ be the injective hull of $R/\m$. For every finitely generated submodule $N$ of $E(R/\m)$, $N$ is $p^n$-torsion for some $n$, thus $N \to B'\otimes_R N$ can be identified with $N \to B'/p^n \otimes_{R/p^n} N$. Hence $N\to B'\otimes_R N$ is injective for every such $N$ by the purity of $R/p^n\to B'/p^n$. By taking a direct limit for all such $N$, we find that $E(R/\m)\to B'\otimes_R E(R/\m)$ is injective for every $\m$. Thus $R\to B'$ is pure thanks to \cite[Proposition 6.11]{HochsterRobertsRingsOfInvariants} (\cf \autoref{lem:injective_E_purity})  as wanted.

{In the \claperfdinj case, notice, we consider the map $H^i_{\m}(B) \to H^i_{\m}(B) \otimes_B B'$.  Just as above, since $H^i_{\m}(B)$ is a colimit of $p^n$-torsion modules and $B/p^n\to B'/p^n$ is faithfully flat and hence pure, we see that $H^i_{\m}(B) \to H^i_{\m}(B) \otimes_B B'$ is injective. But as the functor $N\mapsto N\otimes^{\myL}_BB'$ is $t$-exact on $D_{p\text{-tor}}(B)$ by the $p$-complete flatness of $B'$ over $B$, thus we have  
$$H_\m^i(B')= h^i(\myR\Gamma_\m(B)\otimes^\mathbf{L}_BB')= H_\m^i(B)\otimes_BB'.$$
The result follows.
}
\end{proof}

\begin{lemma}
\label{lem:Puremap}
Suppose $R \to S$ is a pure map of Noetherian rings.  If $p$ lies inside their Jacobson radicals and $S$ is \claperfdpure (resp. \absperfdpure) then so is $R$. 

Furthermore, the same statements hold for \claperfdinj and \absperfdinj if $S$ is a finite $R$-module.
\end{lemma}
\begin{proof}
First suppose that $S$ is \claperfdpure.  Then by definition, there is a perfectoid $S$-algebra $B$ such that $S \to B$ is pure. Hence the composition  $R\to S\to B$ is pure as required. Next suppose that we are in the \absperfdpure case.  Then $S\to S_{\perfd}$ is pure.  Therefore $R\to S_{\perfd}$ is pure by \autoref{lem:composition_is_pure}.  But we have a factorization $R\to R_{\perfd}\to S_{\perfd}$ and so $R\to R_{\perfd}$ is pure by \autoref{lem:first_factor_is_pure}. 

For the \claperfdinj (resp.\ \absperfdinj) case, choose a maximal ideal $\fram$ of $R$ with the finitely many $\fram_j$ maximal ideals of $S$ lying over $\fram$. It is easy to see that $H^i_{\fram}(K)\cong \oplus_jH^i_{\fram_j}(K)$ for every $K\in\myD(S)$ and every $i$. We immediately see that the composition
\[
H^i_{\fram}(R) \to H^i_{\fram}(S) = \oplus_j H^i_{\fram_j}(S) \hookrightarrow \oplus_j H^i_{\fram_j}(K) = H^i_{\fram}(K)
\]
is injective, where $K=B$ a perfectoid $S$-algebra in the \claperfdinj case, and $K=S_{\perfd}$ in the \absperfdinj case. In the former case we are done since $B$ is also a perfectoid $R$-algebra. In the latter case, note that we have a factorization $R\to R_{\perfd}\to S_{\perfd}$, we obtain injectivity of $H_\m^i(R)\to H_\m^i(R_\perfd)$ as wanted. 
\end{proof}

\begin{remark}
It is not true that pure subrings of \claperfdinj (resp.\ \absperfdinj) rings are \claperfdinj (resp.\ \absperfdinj). This is not true even in characteristic $p$ (i.e., for $F$-injective rings, see \autoref{rmk:PerfdPureInjectiveCharp}), for example see \cite{WatanabeFRationalityOfCertainReesAndCounterExamplesToBoutot}. 
\end{remark}

\subsection{Completion and localization}

We next show that, for Noetherian local rings, \claperfdpure (resp. \absperfdpure) and \claperfdinj (resp. \absperfdinj) are preserved under $\m$-adic completion. 

\begin{lemma}
\label{lem:pfdpurecompletion}
    Suppose $(R, \m)$ is a Noetherian local ring of residue characteristic $p>0$. Then $R$ is \claperfdpure if and only if the $\m$-adic completion $R^{\wedge\m}$ is \claperfdpure, in which case $R^{\wedge\m}\to B$ splits for some perfectoid $R^{\wedge\m}$-algebra $B$.

    {Similarly, $R$ is \claperfdinj if and only if $R^{\wedge\m}$ is.}
\end{lemma}
\begin{proof}
In the \claperfdpure case, one direction follows from \autoref{lem:Puremap} since $R\to R^{\wedge\m}$ is faithfully flat and in particular pure. 
For the other implication, suppose now that $B$ is a perfectoid $R$-algebra so that $R \to B$ is pure.
By \autoref{rmk:make B aic}, we may enlarge $B$ to assume that every element of $R$ has a compatible system of $p$-power roots in $B$. Since $R \to B$ is pure, so is the $\m$-adic completion map $R^{\wedge\m} \to B^{\wedge\m}$ (since $E:=E(R/\m)\to E\otimes_R B\cong E\otimes_{R^{\wedge\m}} B^{\wedge\m}$ is injective). Now by \cite[Proposition 2.1.11 (e)]{KestutisScholzePurityFlatCohomology}, the $\m$-adic completion $B^{\wedge\m}$ remains perfectoid. For the last conclusion, simply note that for a Noetherian complete local ring $R$, by Matlis duality, a map $R\to B$ is pure if and only if it splits. 

In the \claperfdinj case, by \autoref{rmk:make B aic} we may assume $B$ admits a compatible system of $p$-power roots for all elements of $R$, and the $\m$-adic completion $B^{\wedge\m}$ agrees with the derived $\m$-adic completion (since $\m$ is finitely generated),
and it is still perfectoid (see \cite[Proposition 2.1.11]{KestutisScholzePurityFlatCohomology}). Now it suffices to observe that
\[H^i_{\fram}(B) \cong H^i(\myR\Gamma_\m(B)) \cong H^i(\myR\Gamma_\m(B^{\wedge\m}))\cong H^i_{\fram}(B^{\wedge\m}).\qedhere \]
\end{proof}

\begin{lemma}
\label{lem:abspurecompletion}
    Suppose $(R, \m)$ is a Noetherian local ring of residue characteristic $p>0$. Then $R$ is \absperfdpure (resp. \absperfdinj\!\!) if and only if the $\m$-adic completion $R^{\wedge\m}$ is \absperfdpure (resp. \absperfdinj\!\!).
\end{lemma}
\begin{proof}

First suppose that $R^{\wedge\m}$ is \absperfdpure (resp. \absperfdinj).  We have a diagram 
$$
\xymatrix{
R\ar[r]\ar[d] & R_{\perfd}\ar[d]\\
R^{\wedge\m}\ar[r]& R^{\wedge\m}_{\perfd}
}
$$
such that the lower composition is pure (resp. injective on local cohomology). Hence the top row is pure (resp. injective on local cohomology) by \autoref{lem:first_factor_is_pure} and thus $R$ is \absperfdpure (resp. \absperfdinj).

Now suppose that $R$ is \absperfdpure (resp. \absperfdinj), so that $R\to R_{\perfd}$ is pure (resp. injective on local cohomology).  Then $R^{\wedge\m}\to (R_{\perfd})^{\wedge\m}$ is $\fram$-completely pure (resp. injective on local cohomology), and so is pure (resp. injective on local cohomology) since $R$ is Noetherian by \autoref{lem:p-completely-pure-implies-pure}.  But $(R_{\perfd})^{\wedge\m}\simeq (R^{\wedge\m}_{\perfd})^{\wedge\m}$ by \autoref{prop:perfds_completion_maximal_ideal}, and so ${R}^{\wedge\m}\to (R^{\wedge\m}_{\perfd})^{\wedge\m}$ is pure (resp. injective on local cohomology).
\end{proof}

We next show that \claperfdpure (resp. \absperfdpure) and \claperfdinj (resp. \absperfdinj) can be checked at localizations at the maximal ideals.

\begin{lemma}
\label{lem:checklocalrings}
Let $R$ be a Noetherian ring with $p$ in its Jacobson radical. Then $R$ is perfectoid pure {(resp. \claperfdinj)} if and only if $R_\m$ is perfectoid pure {(resp. \claperfdinj)} for all maximal ideals $\m$ of $R$. 
\end{lemma}
\begin{proof}
We first handle the \claperfdpure case. Suppose $R$ is \claperfdpure, then $R\to B$ is pure for some perfectoid ring $B$. Then $R_\m \to B_\m$ is pure, i.e., $E(R/\m)\otimes_R R_\m \to E(R/\m) \otimes_R B_\m$ is injective where $E(R/\m)$ is the injective hull of $R/\m$. But as $E(R/\m)$ is $p^{\infty}$-torsion, we know that $E(R/\m)\otimes_R B_\m \cong E(R/\m)\otimes_R (B_\m)^{\wedge p}$ where $(B_\m)^{\wedge p}$ is the $p$-adic completion of $B_\m$, which is perfectoid by \cite[Example 3.8]{BhattIyengarMaRegularRingsPerfectoid}. Thus $R_\m\to (B_\m)^{\wedge p}$ is pure and thus $R_\m$ is \claperfdpure. Conversely, suppose $R_\m$ is \claperfdpure, then $(R_\m)^{\wedge \m}$ is \claperfdpure by \autoref{lem:pfdpurecompletion}, i.e., $(R_\m)^{\wedge \m}\to B(\m)$ is split for some perfectoid ring $B(\m)$. Then $\prod_\m B(\m)$ is perfectoid by \cite[Example 3.8]{BhattIyengarMaRegularRingsPerfectoid} and $\prod_\m (R_\m)^{\wedge \m} \to \prod_\m B(\m)$ is split. Since $R\to \prod (R_\m)^{\wedge \m}$ is faithfully flat (here we are using that $R$ is Noetherian), it follows that the composition $R\to \prod_\m(R_\m)^{\wedge \m}\to \prod_\m B(\m)$ is pure and thus $R$ is \claperfdpure. 

The proof in the \claperfdinj case is similar: first note that we have $H_\m^i(B)\cong H_\m^i(B_\m)\cong H_\m^i((B_\m)^{\wedge p})$ (the second isomorphism follows as the $p$-adic completion agrees with the derived $p$-completion as $B$ and thus $B_\m$ has bounded $p^\infty$-torsion). So if $H_\m^i(R)\to H_\m^i(B)$ is injective for all $i$, then $H_\m^i(R_\m)\to H_\m^i((B_\m)^{\wedge p})$ is injective and thus $R_\m$ is \claperfdinj. Conversely, assuming $R_\m$ is \claperfdinj and using the same notation as in the \claperfdpure case, we have $H_\m^i(R)\to H_\m^i(\prod_\m B(\m))$ is injective for every maximal ideal $\m$ and every $i$ since this map factors the injective map $H_\m^i(R)\to H_\m^i(B(\m))$. Thus $R$ is \claperfdinj as wanted.
\end{proof}

\begin{lemma}
\label{lem:checklocalringsAbsPure}
Let $R$ be a Noetherian ring with $p$ in its Jacobson radical. Then $R$ is \absperfdpure (resp. \absperfdinj) if and only if $R_\m$ is \absperfdpure (resp. \absperfdinj) for every maximal ideal $\m$ of $R$.
\end{lemma}
\begin{proof}
We first assume $R$ is \absperfdpure, i.e., $R\to R_\perfd$ is pure, then $R_\m\to (R_\perfd)_\m$ is pure and thus $R_\m\to (R_\perfd)_\m^{\wedge p}\cong (R_\m)_\perfd$ is $p$-completely pure, where the isomorphism follows from \autoref{lem:perfd_localization}. It follows that $R_\m\to (R_\m)_\perfd$ is pure by \autoref{lem:p-completely-pure-implies-pure}. Conversely, suppose $R_\m$ is \absperfdpure for every maximal ideal $\m$. Then by \autoref{lem:abspurecompletion}, we know that $(R_\m)^{\wedge\m}$ is \absperfdpure. Consider the composition: 
\[R\to \prod_{\m}(R_\m)^{\wedge\m}\to \prod_\m ((R_\m)^{\wedge\m})_\perfd,\]
where the first map is faithfully flat and the second map is split in $\myD(R)$. It follows that the composition is pure in $\myD(R)$ by \autoref{lem:composition_is_pure} and since this map factors through $R_\perfd$, we have that $R\to R_\perfd$ is pure by \autoref{lem:first_factor_is_pure}. 

The proof in the \absperfdinj case is even easier by noting that 
$$H_\m^i(R_\perfd)\cong H_\m^i((R_\perfd)_\m)\cong H_\m^i((R_\perfd)_\m^{\wedge p})\cong H_\m^i((R_\m)_\perfd)$$
where the second isomorphism can be deduced from \cite[\href{https://stacks.math.columbia.edu/tag/0A6W}{Tag 0A6W}]{stacks-project} and the last isomorphism follows from \autoref{lem:perfd_localization}. We leave the details to the readers. 
\end{proof}

Next, we show that \claperfdpure (resp. \absperfdpure) and \claperfdinj (resp. \absperfdinj) are preserved under localization. We need a lemma.

\begin{lemma}
\label{lem:DualizingComplex}
 Suppose $R$ is a Noetherian ring with a dualizing complex. Suppose that $M \to N$ is a map in $D(R)$ with $M$ a perfect complex such that 
 \[
    H^i_{\fram}(M) \to H^i_{\fram}(N)
 \]
 injects for every $i$ and every maximal ideal in $\fram$.  Then for every $Q \in \Spec R$ we have that 
 \[
    H^i_{Q}(M_Q) \to H^i_{Q}(N_Q)
 \]
 injects.  

 In particular, if $R$ with $p$ in the Jacobson radical is \claperfdinj (with $B$ exhibiting this property), respectively if it is \absperfdinj, then 
 \[
    H^i_Q(R_Q) \to H^i_Q(B_Q), \;\;\;\text{ respectively } \;\;\;H^i_Q(R_Q) \to H^i_Q((R_{\perfd})_Q)
 \]
 injects for every $Q \in \Spec R$.
\end{lemma} 
\begin{proof}
Write $N$ as a filtered colimit of compact objects $N_{\lambda}$ (that is, perfect complexes), $N = \colim N_{\lambda}$ where the colimit is taken in the derived infinity category, and where we fix maps $M \to N_{\lambda}$ colimiting to $M \to N$.
Note that $N_Q = \colim (N_{\lambda})_Q$.  It thus suffices to show that each $H^i_Q(M_Q) \to H^i_Q((N_{\lambda})_Q)$ injects.  Let $\fram \supseteq Q$ be a maximal ideal.  We are given that $H^i_{\fram}(M_{\fram}) \to H^i_{\fram}((N_{\lambda})_{\fram})$ injects. Local duality, localization, and then local duality again, completes the proof.

For the application, take $M = R$ and $N = B$ respective $N = R_{\perfd}$.
\end{proof}

\begin{lemma}
\label{lem:PerfdPureLocalizes}
Let $R$ be a Noetherian ring with $p$ in its Jacobson radical. If $R$ is \claperfdpure (resp. \claperfdinj) and $Q\in\Spec(R)$ such that $p\in Q$, then $R_Q$ is \claperfdpure (resp. \claperfdinj). Consequently, $W^{-1}R$ is \claperfdpure (resp. \claperfdinj) for any multiplicative system $W$ such that $p$ is in the Jacobson radical of $W^{-1}R$.
\end{lemma}
\begin{proof}
In the \claperfdpure case, note that the composition $R_Q\to B_Q \to (B_Q)^{\wedge p}$ is pure and $(B_Q)^{\wedge p}$ is perfectoid (see the proof of \autoref{lem:checklocalrings}) and so $R_Q$ is \claperfdpure. In the \claperfdinj case, after localizing at a maximal ideal containing $Q$ and applying \autoref{lem:checklocalrings} we may assume $(R,\m)$ is local. Now by \autoref{lem:pfdpurecompletion} we know that $R^{\wedge\m}$ is \claperfdinj and admits a dualizing complex. Choose $Q'\in \Spec(R^{\wedge\m})$ a minimal prime of $QR^{\wedge\m}$ and we have
$$H_Q^i(R_Q)\to H_Q^i((R^{\wedge\m})_{Q'}) \to H_Q^i(B_{Q'})\cong H_Q^i((B_{Q'})^{\wedge p}).$$
The first map above is injective by faithful flatness of $R_Q\to (R^{\wedge\m})_{Q'}$, the second map is injective by \autoref{lem:DualizingComplex} and the isomorphism follows since $p$ completion agrees with derived $p$-completion here (as $B$ and thus $B_{Q'}$ has bounded $p^\infty$-torsion). Thus the composition is injective, since $(B_{Q'})^{\wedge p}$ is perfectoid by \cite[Example 3.8]{BhattIyengarMaRegularRingsPerfectoid}, $R_Q$ is \claperfdinj as wanted. 

The last conclusion follows from \autoref{lem:checklocalrings}. 
\end{proof}

\begin{lemma}
    \label{lem.absPerfdPureLocalizes}
Let $R$ be a Noetherian ring with $p$ in its Jacobson radical. If $R$ is \absperfdpure (resp. \absperfdinj) and $Q\in\Spec(R)$ such that $p\in Q$, then $R_Q$ is \absperfdpure (resp. \absperfdinj). Consequently, $W^{-1}R$ is \absperfdpure (resp. \absperfdinj) for any multiplicative system $W$ such that $p$ is in the Jacobson radical of $W^{-1}R$.
\end{lemma}
\begin{proof}
By \autoref{lem:perfd_localization}, we have $(R_Q)_{\perfd}\simeq (R_{\perfd})_Q^{\wedge p}$.  
Now suppose that $R\to R_{\perfd}$ is pure, for which it follows that $R_Q\to (R_{\perfd})_Q$ is pure.  Then $R_Q\otimes^{\myL}\mathbb{Z}/p^n\to (R_{\perfd})_Q\otimes^{\myL}\mathbb{Z}/p^n$ is pure by \autoref{lem:tensor-product-food}, and so $(R_Q)^{\wedge p}\to (R_{\perfd})_Q^{\wedge p}$ is $p$-completely pure and hence pure by \autoref{lem:p-completely-pure-implies-pure}.  Thus by \autoref{lem:composition_is_pure} the composition $$R_Q\to (R_Q)^{\wedge p} \to (R_{\perfd})_Q^{\wedge p}\simeq (R_Q)_{\perfd}$$ is pure since $R_Q\to (R_Q)^{\wedge p}$ is faithfully flat as $R$ is Noetherian ring and $p\in Q$.

Now suppose that $R$ is \absperfdinj. After localizing at a maximal ideal containing $Q$ and applying \autoref{lem:checklocalringsAbsPure}, we may assume $(R,\m)$ is local. Now by \autoref{lem:pfdpurecompletion} we know that $R^{\wedge\m}$ is \absperfdinj and admits a dualizing complex. Choose $Q'\in \Spec(R^{\wedge\m})$ a minimal prime of $QR^{\wedge\m}$ and we have
$$H_Q^i(R_Q)\to H_Q^i((R^{\wedge\m})_{Q'}) \to H_Q^i((R^{\wedge\m}_\perfd)_{Q'})\cong H_Q^i\big(((R^{\wedge\m})_{Q'})_\perfd\big).$$
The first map above is injective by faithful flatness of $R_Q\to (R^{\wedge\m})_{Q'}$, the second map is injective by \autoref{lem:DualizingComplex} and that $R^{\wedge\m}$ is \absperfdinj, and the isomorphism follows from \autoref{lem:perfd_localization}. Thus the composition is injective, since the map factors through $H_Q^i((R_Q)_\perfd)$, we know that $H_Q^i(R_Q)\to H_Q^i((R_Q)_\perfd)$ is injective and thus $R_Q$ is \absperfdinj as wanted. 

The last conclusion follows from \autoref{lem:checklocalrings}. 
\end{proof}

\subsection{Smooth ring extensions}
In this subsection we show the behavior of perfectoid purity under smooth extensions. We begin with some lemmas.

\begin{lemma}
    \label{lem.EtaleMapExtends}
    Suppose $R \to S$ is an \'etale map between Noetherian rings that have $p$ in their Jacobson radicals. Suppose $R$ is \claperfdpure, \claperfdinj, \absperfdpure or \absperfdinj, then so is $S$.
\end{lemma}
\begin{proof}
    If $R \to B$ is a pure map such that $B$ is perfectoid, then the base change $S \to B \widehat \otimes_R S$ is also pure, thus $S$ is perfectoid pure since $B \widehat \otimes_R S$ is a perfectoid ring. 
    The analogous statement for \absperfdpure follows by \autoref{prop:formally_etale_perfd}.  The statements for perfectoid injective and lim-perfectoid injective follow similarly since $H^i_{\fram}(-) \otimes_R S = H^i_{\fram}(- \otimes_R S)$ as $R \to S$ is flat.
\end{proof}

\begin{lemma}
    \label{clm.BhattSmoothnessClaim}
    Suppose $R$ is a perfectoid ring. Let $G = \mathbf{G}_m^n \times \mathrm{Spf}(R)$ and set $T=\mathcal{O}(G)$, so $T= R[x_1^{\pm 1}, \dots, x_n^{\pm 1}]^{\wedge p}$. Then $T \to T_{\perfd}$ admits a $T$-module splitting which is functorial in the map $R \to T$ (and, importantly, compatible with base change in $R$).
\end{lemma}
\begin{proof}
We shall argue geometrically. Write $G' = \lim_p G$ for the naive perfection of $G$ (the inverse limit of multiplication by $p$ on $G$). Note that multiplication by $p^m$ on $G$ is faithfully flat with kernel the linearly reductive group scheme $\mu_{p^m}^n$. Taking inverse limits in $m$ shows that the natural map $G' \to G$ is faithfully flat with kernel $K = \bZ_p(1)  = \lim \mu_{p^m}$ also being linearly reductive. Since $G'$ is perfectoid by construction, this map factors as a $G' \to G_{\perfd} \to G$ of group stacks. Note that $K$ in $G'$ maps to 0 under the composite by construction. So we can pass to the quotient by $K$ to obtain a map $G = G'/K \to G_{\perfd}/K$ splitting the map $G_{\perfd}/K \to G$. Passing to rings, this shows that
    \[ 
        T := \cO(G) \to \cO(G_{\perfd}/K) 
    \]
    admits a section (even as rings). Using the fact that $K$ is linear reductive, we also know that 
    \[
        \cO(G_{\perfd}/K) \to \cO(G_{\perfd}) = T_{\perfd}
    \]
    admits a natural module splitting (in fact, representations of $K$ are $\bZ[1/p]/\bZ$-graded modules, and taking cohomology just means taking the degree 0 summand). Combining the two shows that $T \to T_{\perfd}$ admits a natural $T$-module splitting, as wanted.
\end{proof}

The next lemma is well-known, we include a short proof for completeness as we cannot find a good reference.

\begin{lemma}
\label{lem: filtered colimits commutes cosimplicial}
Let $\{C_i^\bullet\}_{i\in I}$ be a filtered diagram of cosimplicial objects in $D(R)$ so that $C_i^n\in D^{\geq 0}$ for all $i, n$. Then the natural map 
$$\varinjlim_{i\in I}\lim_{\Delta} C_i^\bullet \to \lim_{\Delta}\varinjlim_{i\in I} C_i^\bullet$$
is an isomorphism.
\end{lemma}
\begin{proof}
It is enough to check this isomorphism after taking $H^k(-)$ for each $k\geq 0$. But after taking $H^k(-)$, on both sides we could replace $\lim_{\Delta}$ by $\lim_{\Delta_{\leq k+1}}$: this is because 
$$\text{fib} \left(\lim_{\Delta} C_i^\bullet \to \lim_{\Delta_{\leq k+1}}C_i^{\bullet \leq k+1}\right)$$
lives in $D^{> k+1}$ by our assumption that $C_i^n\in D^{\geq 0}$ for all $i, n$. Thus the isomorphism follows since filtered colimits commutes with finite limits.
\end{proof}

\begin{proposition}
    \label{prop.SmoothMapExtends}
    Suppose $R \to S$ is a finite type map between Noetherian rings with $p$ in the Jacobson radical of $R$.  Suppose $\frq$ is a prime ideal of $S$ that contains $p$ such that $\Spec (S) \to \Spec (R)$ is smooth at $\frq$.  If $R$ is 
    \begin{enumerate}
        \item \claperfdpure, respectively 
        \item \claperfdinj,
        \item \absperfdpure, or \label{prop.SmoothMapExtends.absperfdpure}
        \item \absperfdinj
    \end{enumerate}
    then so is $S_{\frq}$.
\end{proposition}
\begin{proof}
By \cite[\href{https://stacks.math.columbia.edu/tag/054L}{Tag 054L}]{stacks-project}, we may assume that $(R,\m)$ is local, $S$ is \'etale over $R[x_1,\dots,x_n]$, and $\mathfrak{q}\in\Spec(S)$ that contracts to $\m$. 

\begin{enumerate}
\item Using \autoref{lem.EtaleMapExtends}, we can reduce to the case that $S = R[x_1, \dots, x_n]$ and that $\mathfrak{q}\in\Spec(S)$ contracts to $\m$.
Suppose $R$ is \claperfdpure, i.e., there exists a perfectoid $B$ such that $R\to B$ is pure.  It is easy to see that $S_{\frq} \to B[x_1^{1/p^{\infty}}, \dots, x_n^{1/p^{\infty}}]_{\frq} =: B'$ is also pure. Since $B'^{\wedge p}$ is perfectoid and we have that $S_{\frq} \to B'^{\wedge p}$ is pure (as the map remains injective after tensoring with $E_{S_{\frq}}$, the injective hull of $S_{\frq}/\frq S_{\frq}$), it follows that $S_\frq$ is perfectoid pure.

\item Again, we may assume that $S = R[x_1, \dots, x_n]$ and $\mathfrak{q}\in\Spec(S)$ contracts to $\m$.  Suppose $R$ is perfectoid injective, i.e., there exists a perfectoid $B$ such that $H_\m^i(R)\to H_\m^i(B)$ is injective for all $i$. We first claim the following.
\begin{claim}
\label{clm:LC injection adjoing variables}
Suppose $R\to C$ is a map in $D(R)$ such that $H_\m^i(R)\to H_\m^i(C)$ is injective for all $i$. Then, for each maximal ideal $\m'$ of $S$ that contracts to $\m$, we have that $H_{\m'}^i(S)\to H_{\m'}^i(C\otimes_RS)$ is injective for all $i$.
\end{claim}
\begin{proof}[Proof of Claim]
By an obvious induction we may assume that $S=R[x]$ and $\m'=\m+(f(x))$ where $f(x)\in R[x]$ is a monic polynomial whose image in $(R/\m)[x]$ is irreducible. Now the morphism $R[y]_{\m + (y)}\to R[x]_{\m'}$ sending $y$ to $f(x)$ is faithfully flat (for example, by using \cite[\href{https://stacks.math.columbia.edu/tag/00ML}{Tag 00ML}]{stacks-project}), and base change along this map induces a commutative diagram 
\[
\xymatrix{
H_{\m+(y)}^i(R[y]) \ar[r] \ar[d] & H_{\m+(y)}^i(C\otimes_RR[y]) \ar[d] \\
H_{\m'}^i(R[x]) \ar[r] & H^i_{\m'}(C\otimes_RR[x]) 
}.
\]
Since $H_\m^i(R)\to H_\m^i(C)$ is injective for all $i$, we have that $H_{\m+(y)}^i(R[y]) \to H_{\m+(y)}^i(C\otimes_RR[y])$ is injective for all $i$ (note that $H_{\m+(y)}^i(R[y])\cong H_\m^{i-1}(R)[y^{-1}]$ and similarly for $H_{\m+(y)}^i(C\otimes_RR[y])$). Thus by flatness of the base change, we have that $H_{\m'}^i(R[x])\to H_{\m'}^i(C\otimes_RR[x])$ is injective for all $i$.
\end{proof}
Since $B[x_1,\dots,x_n]\to B[x_1^{1/p^{\infty}}, \dots, x_n^{1/p^{\infty}}] =: B'$ is faithfully flat, together with \autoref{clm:LC injection adjoing variables} (applied to $C=B$) we have that 
$$H_{\m'}^i(S_{\m'}) \cong H_{\m'}^i(S)\to H_{\m'}^i(B')\cong H^i_{\m'}(B'^{\wedge p}_{\m'})$$
is injective for all $i$. Thus, as $B'^{\wedge p}_{\m'}$ is perfectoid, $S_{\m'}$ is perfectoid injective for each maximal ideal $\m'$ of $S$ that contracts to $\m$. Finally, since $\frq$ contracts to $\m$, we can pick such a maximal ideal $\m'$ that contains $\frq$. By \autoref{lem:PerfdPureLocalizes}, $S_{\frq}$ is perfectoid injective. 

\item Suppose $R$ is lim-perfectoid pure. Similar to the reductions above, we may assume that our point of interest $\frq$ lies in $\mathbf{G}_m^n \subset \mathbf{A}^n$ over $\mathrm{Spf}(R)$  and thus we may assume $S=R[x_1^{\pm 1},...,x_n^{\pm 1}]$. By base change, it is enough to explain that $S' := R_{\perfd} \widehat{\otimes}_R S \to S'_{\perfd}$ is pure. In this case, we shall actually explain why it is naturally split over $S'$. We first note that we can write $R_\perfd = \varprojlim B^\bullet$ where $B^\bullet$ is a cosimplicial perfectoid $R$-algebra. To see this, we take a $p$-complete arc cover $R\to B^0$ with $B^0$ perfectoid (for example, see \cite[Lemma 2.2.3]{KestutisScholzePurityFlatCohomology}). Then we note that, since perfectoidization is an arc-sheaf (see \autoref{prop:h-sheafification}), we have 
\[R_{\perfd} \cong \varprojlim \left((B_0)_\perfd \rightrightarrows (B_0\widehat{\otimes}^{\myL}_RB_0)_{\perfd} \, \substack{\longrightarrow\\[-1em] \longrightarrow \\[-1em] \longrightarrow} \cdots \right)\]
and that all the terms in the cosimplicial limit above are perfectoid rings since they are perfectoidizations of semi-perfectoid rings (so the right hand side above provides us the desired $\varprojlim B^\bullet$). Now as $S'$ is a $p$-completely free $R_\perfd$-module (as $S$ was a $p$-completely free $R$-module), the functor $- \widehat{\otimes}_{R_\perfd} S'$ commutes with cosimplicial limits of coconnective $R_\perfd$-complexes: indeed, the case of finite free modules is clear, and one passes to the limit by \autoref{lem: filtered colimits commutes cosimplicial}. It follows that 
$$S' \simeq \varprojlim (B^\bullet \widehat{\otimes}_{R_\perfd} S') = \varprojlim(B^\bullet \widehat{\otimes}_R S).$$ 
As the composition $S' \simeq \varprojlim(B^\bullet \widehat{\otimes}_R S) \to \varprojlim (B^\bullet \widehat{\otimes}_R S)_\perfd$ factors over $S' \to S'_\perfd$, to show that the latter is split, it suffices to show that $\varprojlim(B^\bullet \widehat{\otimes}_R S) \to \varprojlim (B^\bullet \widehat{\otimes}_R S)_\perfd$ is split. But for any perfectoid $R$-algebra $B$, the map $B \widehat{\otimes}_R S \to (B \widehat{\otimes}_R S)_\perfd$ admits a canonical splitting by \autoref{clm.BhattSmoothnessClaim}; taking inverse limits over $B^\bullet$ then gives the desired splitting.

\item Suppose $R$ is \absperfdinj. As in part (c), we may assume $S=R[x_1^{\pm 1},...,x_n^{\pm 1}]$. While \autoref{clm:LC injection adjoing variables} was proven for $S = R[x_1, \dots, x_n]$, it also applies to $S=R[x_1^{\pm 1},...,x_n^{\pm 1}]$.  Indeed, any maximal ideal of $R[x_1^{\pm 1}, \dots, x_n^{\pm 1}]$ contracting to $\m \subseteq R$ comes from a prime ideal $Q \subseteq R[x_1, \dots, x_n]$ contracting to $\m$.  But $Q$ is then contained in a maximal ideal of $R[x_1, \dots, x_n]$ contracting to $\m$.    We thus see that $S=R[x_1^{\pm 1},...,x_n^{\pm 1}]$ can also be used in the statement of \autoref{clm:LC injection adjoing variables} thanks to \autoref{lem:DualizingComplex}.

Therefore, using \autoref{clm:LC injection adjoing variables} applied to $C=R_{\perfd}$ and our $S$, the map $S \to S' := R_{\perfd} \widehat{\otimes}_R S$ is injective on local cohomology over all maximal ideals in $S$ containing $p$, so it is enough to explain the same property for the map $S' \to S'_{\perfd}$. But this map is pure as explained in part (c), so we win.
\end{enumerate}
\end{proof}

    \begin{lemma}
        \label{clm.FunctorialPerfectoidForSmoothExtensionsOfPerfectoid}
        \label{cor.IndSmoothMapExtends.clm.FunctorialAssignmentForPerfectoidForSmoothExtensions}
        Suppose $A$ is a perfectoid ring.  Then for each smooth $A$-algebra $S$, there is a {$p$-complete} faithfully flat map $S \to B(S)$ where $B(S)$ is perfectoid.  Furthermore, this assignment is functorial in $A \to S$.
    \end{lemma}
    \begin{proof}
        Pick $E$ a finite set of generators of $S$ as an $A$-algebra.  Consider the map $S \to T(E) := S[x^{1/p^{\infty}}]^{\wedge p}_{x \in E}$ with semiperfectoid target.  Set $F(E)$ to be the perfectoidization of $T(E)$, which is a perfectoid ring by \cite[Theorem 7.4]{BhattScholzepPrismaticCohomology}.  

        We next show that $S \to T(E) \to F(E)$ is $p$-complete faithfully flat. To see this, we note that by \cite[\href{https://stacks.math.columbia.edu/tag/054L}{Tag 054L}]{stacks-project}, we may assume that $S$ is an \'etale extension of $A[z_1,\dots,z_n]$ for some $n$. We have $S \to S_\infty := A[z_1^{1/p^\infty}, \dots, z_n^{1/p^\infty}]^{\wedge p} \widehat{\otimes}_{A[z_1,\dots,z_n]}S$ is $p$-complete faithfully flat map of perfectoid rings. We have a commutative diagram 
        \[\xymatrix{
S \ar[r] \ar[d]  & F(E) \ar[d]\\
 S_\infty \ar[r] & F(E)[z_1^{1/p^\infty}, \dots, z_n^{1/p^\infty}]^{\wedge p}        
        }.\]
        The vertical maps are $p$-complete faithfully flat, and the bottom row is $p$-complete faithfully flat by Andr\'{e}'s flatness lemma (see \cite[Theorem 7.14]{BhattScholzepPrismaticCohomology}), because $F(E)$ is adjoining compatible system of $p$-power roots of certain elements to $S_\infty$. It follows that the top row is also $p$-complete faithfully flat. We now set $B(S) := (\colim_E F(E))^{\wedge p}$ where the colimit runs over inclusions of sets $E \subseteq E'$.  
    \end{proof}

\begin{corollary}
    \label{cor.IndSmoothMapExtends}
    Suppose $R \to S$ is a regular map of rings with $p$ in their Jacobson radicals.  If $R$ is 
    \begin{enumerate}
        \item \claperfdpure, respectively  \label{cor.IndSmoothMapExtends.perfdPure}
        \item \claperfdinj, \label{cor.IndSmoothMapExtends.perfdInj}
        \item \absperfdpure, or \label{cor.IndSmoothMapExtends.LimPerfdPure}
        \item \absperfdinj \label{cor.IndSmoothMapExtends.LimPerfdInj}
    \end{enumerate}
    then so is $S$.
\end{corollary}
\begin{proof}
    We may write $S = \colim S_{\lambda}$ as a filtered colimit of smooth $R$-algebras by Popescu's theorem \cite{PopescuGeneralNeronDesingularization,PopescuLetterToEditor}, \cite[\href{https://stacks.math.columbia.edu/tag/07GB}{Tag 07GB}]{stacks-project}.  
    Before getting into the details, we begin with a claim.  

    \begin{itemize}
        \item[\autoref{cor.IndSmoothMapExtends.perfdPure}]  Fix $A$ to be a perfectoid ring such that $R \to A$ is pure.  We then have that the map $A \otimes_R {S_{\lambda}} \to B(A \otimes_R S_{\lambda})$ is $p$-completely faithfully flat by \autoref{cor.IndSmoothMapExtends.clm.FunctorialAssignmentForPerfectoidForSmoothExtensions}.  Hence $S_{\lambda} \to A \otimes_R {S_{\lambda}} \to B(A \otimes_R S_{\lambda})$ is $p$-completely pure.  Taking colimit we see that 
        \[
            S = \colim_{\lambda} S_{\lambda} \to \colim_{\lambda} B(A \otimes_R S_{\lambda}) \to \widehat{\colim}_\lambda B(A\otimes_RS_\lambda)
        \]
        is $p$-completely pure and hence pure since $S$ is Noetherian.  
        \item[\autoref{cor.IndSmoothMapExtends.perfdInj}]  Without loss of generality, we may assume that $(R, \fram)$ and $(S, \frn)$ are local.  Additionally, we may pick $\frn_{\lambda} \in \Spec S_{\lambda}$ the image of $\frn \in \Spec S$.  Replacing $S_{\lambda}$ by $S_{\lambda, \frn_{\lambda}}$ we may assume that each $S_{\lambda}$ is local and $S_{\lambda} \to S$ is a local map, and that $S_{\lambda}$ is a localization of a smooth $R$-algebra.   
        
        Fix $A$ to be a perfectoid ring such that $H^i_{\fram}(R) \to H^i_{\fram}(A)$ is injective.  By \autoref{clm:LC injection adjoing variables} (followed by an \'etale extension which is harmless) we have that each $H^i_{\frn_{\lambda}}(S_{\lambda}) \to H^i_{\frn_{\lambda}}(A \otimes_R S_{\lambda})$ is injective. By \autoref{clm.FunctorialPerfectoidForSmoothExtensionsOfPerfectoid} the map $A \otimes_R S_{\lambda} \to B(A \otimes_R S_{\lambda})$ is $p$-completely faithfully flat, thus the map $H^i_{\frn_{\lambda}}(A \otimes_R S_{\lambda}) \to H^i_{\frn_{\lambda}}(B(A \otimes_R S_{\lambda}))$ is injective.  Composing, we obtain that 
        \[
            H^i_{\frn_{\lambda}}(S_{\lambda}) \to H^i_{\frn_{\lambda}}(B(A \otimes_R S_{\lambda}))
        \]
        injects.  By the functoriality in \autoref{clm.FunctorialPerfectoidForSmoothExtensionsOfPerfectoid} and taking a colimit completes the proof.        
        \item[\autoref{cor.IndSmoothMapExtends.LimPerfdPure}]
        We simply note that $S_\perfd \cong \widehat{\varinjlim}_\lambda (S_\lambda)_\perfd$, using \autoref{def:perfd} and the fact that the functor $\myR\Gamma(\WCart^{\HT},-)$ commutes with colimits \cite[Corollary 3.5.13]{BhattLurieAbsolute}. Now by \autoref{prop.SmoothMapExtends} each $S_\lambda$ is lim-perfectoid pure and thus 
        $$S= \colim_\lambda S_\lambda \to \widehat{\varinjlim}_\lambda (S_\lambda)_\perfd \cong S_\perfd$$ 
        is $p$-completely pure and thus pure since $S$ is Noetherian. 
        \item[\autoref{cor.IndSmoothMapExtends.LimPerfdInj}] Similar to the above, we have $H_\n^i(S) = H_\n^i(\colim_\lambda S_\lambda) \hookrightarrow H_\n^i({\varinjlim}_\lambda (S_\lambda)_\perfd) \cong H_\n^i(S_\perfd)$ where the middle injection follows from \autoref{prop.SmoothMapExtends}. 
    \end{itemize}
\end{proof}

We expect the following stronger result to be true in view of the positive characteristic picture, see \cite{HochsterHunekeFRegularityTestElementsBaseChange,Aberbach2001a,HashimotoCMFinjectiveHoms,AberbachEnescuStructureOfFPure,SchwedeZhangBertiniTheoremsForFSings,DattaMurayamaPermanencePropertiesFinjectivity} or \cite[Chapter 7]{MaPolstraFsingularitesBook}, however we do not see how to prove it.
\begin{conjecture}
    Let $(R,\m)$ be a Noetherian local ring and let $(R, \fram) \subseteq (S, \frn)$ be a flat local homomorphism. Then we have 
    \begin{enumerate}
        \item If $R$ is perfectoid pure and $S/\m S$ is regular (or even Gorenstein $F$-pure), then $S$ is perfectoid pure.
        \item If $R$ is perfectoid injective and $S/\m S$ is Cohen-Macaulay and geometrically $F$-injective, then $S$ is perfectoid injective.   
    \end{enumerate}
\end{conjecture}

Based on the above, the following question is also natural, see \cite{DattaTuckerOpenness} for a partial characteristic $p > 0$ analog.

\begin{question}
    Consider subsets 
    \[
        \begin{array}{c}
            Z = \{ Q \in \Spec R \;|\; p \in Q, R_Q \text{ is \absperfdinj} \} \\
            \text{ and } \\
            Z' = \{ Q \in \Spec R \;|\; p \in Q, R_Q \text{ is \claperfdinj} \}
        \end{array}
    \]
    Is $Z$, respectively $Z'$, open in $V(p)$, the $(p=0)$-fiber?  How is it related to the Du Bois locus if one inverts $p$?  The analogous questions also hold for the \claperfdpure or \absperfdpure loci and their relation to log canonical singularities (at least in the $\bQ$-Gorenstein setting).
\end{question}

\subsection{Cohen-Macaulayness of the perfectoidization in the complete intersection case}

Before we state the next lemma, we fix some notation as follows. Let $k$ be a field of characteristic $p>0$ and let $C_k$ be the unique complete unramified DVR with residue field $k$ (see \cite[\href{https://stacks.math.columbia.edu/tag/0328}{Tag 0328}]{stacks-project}) and fix an inclusion $C_k\to W(k^{1/p^\infty})$. Let $W(k^{1/p^\infty})\to \cO_C$ be a $p$-completed integral extension such that $\cO_C$ is a perfectoid valuation ring.

If $S:= C_k[[x_1,\dots, x_n]]$, then we set 
\begin{equation}
    \label{eq.SInftyDefinition}
    S_{\infty} := (S\widehat{\otimes}_{C_k}\cO_C)[x_1^{1/p^\infty},\dots, x_n^{1/p^\infty}]^{\wedge p},
\end{equation}
where the completed tensor product is $p$-adic. Then $S_\infty$ is a perfectoid ring and we have an induced map $S\to S_\infty$.

\begin{lemma}
    \label{lem:RAperfdsuffices}
    Let $(R, \m, k)$ be a Noetherian complete local ring of residue characteristic $p>0$. Let $\phi$: $S= C_k[[x_1,\dots, x_n]]\to R$ be a map such that $R$ is module-finite over the image of $\phi$ (e.g., $S\twoheadrightarrow R$ or $S\to R$ is a Noether-Cohen-normalization of $R$ when $R$ is a domain). Then the following are equivalent.
        \begin{enumerate}
            \item $R$ is \claperfdpure.
            \item $R\to R^{S_\infty}_\perfd:= (R \otimes_S {S_{\infty}})_{\perfd}$ is pure.
        \end{enumerate}
    {
        Similarly, the following are equivalent.
        \begin{enumerate}
            \item $R$ is \claperfdinj.
            \item $H^i_{\fram}(R) \to H^i_{\fram}(R^{S_\infty}_\perfd)$ is injective for all $i$.
        \end{enumerate}
    }
    \end{lemma}
    \begin{proof}
     Since $R^{S_\infty}_\perfd$ is a perfectoid $R$-algebra, $(b)\Rightarrow(a)$ is trivial {in either case}. We next show $(a)\Rightarrow(b)$. By \autoref{rmk:make B aic}, we may assume that $R\to B$ is pure (respectively, injective on cohomology, in the \claperfdinj case) such that all elements of $R$ have compatible system of $p$-power roots in $B$. We may further replace $B$ by the $\m$-adic completion of $B$ to assume that $B$ is perfectoid and $\m$-adically complete by \cite[Proposition 2.1.11 (e)]{KestutisScholzePurityFlatCohomology} (see the proof of \autoref{lem:checklocalrings}). It follows we have a natural map $S_\infty\to B$: we clearly have $W(k^{1/p^\infty})\to B$ and $C_k[[x_1,\dots, x_n]]\to B$, and since $B$ is $\m$-adically complete, we have a natural map $W(k^{1/p^\infty})[[x_1,\dots,x_n]]\to B$, but since the image of $x_1,\dots,x_n$ all have compatible system of $p$-power roots in $B$ and $B$ is perfectoid, we have a natural map $S_\infty\to B$. It is clear that this map agrees with the canonical map $R\to B$ when restricted to (the image of) $S$. Therefore we have a natural map $R\otimes_SS_\infty\to B$ which induces a natural map $R^{S_{\infty}}_{\perfd} \to B$ since $B$ is perfectoid. Since $R\to B$ is pure {(respectively, induces injections on cohomology)}, we have $R\to R^{S_{\infty}}_{\perfd}$ is pure {(respectively, induces injections on cohomology)} as it factors $R\to B$.
    \end{proof}

\begin{theorem}
\label{thm: lci case}
Let $(R, \m, k)$ be a Noetherian complete local ring of residue characteristic $p>0$. Let $\phi$: $S= C_k[[x_1,\dots, x_n]]\to R$ be a map such that $R$ is module-finite over the image of $\phi$ (e.g., $S\twoheadrightarrow R$ or $S\to R$ is a Noether-Cohen-normalization of $R$ when $R$ is a domain). Suppose $R$ is a complete intersection. Then $R^{S_{\infty}}_{\perfd}$ and $R_{\perfd}$ are Cohen-Macaulay in the sense that 
        \[
            H^i_{\fram}(R^{{S_{\infty}}}_{\perfd}) {= H^i_{\fram}(R_{\perfd})} = 0 
        \]
        for all $i < d = \dim(R)$.

Furthermore, $R$ is \claperfdinj if and only if $R$ is \absperfdinj.
\end{theorem}
\begin{proof}
We first show that $H^i_{\fram}(R^{{S_{\infty}}}_{\perfd})=0$ for all $i<d$. By adjoining new variables to $S$ we may form $S'$ such that $S'\to R$ is surjective. The resulting map $R^{S_{\infty}}_{\perfd}\to R^{S'_{\infty}}_{\perfd}$ is obtained by freely adjoining $p$-power roots of certain elements to $R^{S_\infty}_\perfd$ in the world of perfectoid rings, and is thus $p$-completely faithfully flat by Andr\'{e}'s flatness lemma (see proof of \cite[Theorem 7.14]{BhattScholzepPrismaticCohomology}).  Thus if $\myR\Gamma_{\fram}(R^{{S'_{\infty}}}_{\perfd})\in D^{\geq d}$ then so is $\myR\Gamma_{\fram}(R^{{S_{\infty}}}_{\perfd})$. Hence without loss of generality, we may assume that $S\to R$ is surjective. 

Now $R=S/(f_1,\dots,f_c)$ where $f_1,\dots,f_c$ is a regular sequence. We shall give two proofs\footnote{In fact, the proofs are closely related: the proof of Andr\'e's flatness lemma in \cite[Theorem 7.14]{BhattScholzepPrismaticCohomology} relies on the Hodge--Tate comparison.} that $R^{S_\infty}_{\perfd}$ has vanishing local cohomology in degrees $< d$. In what follows, all completed tensor products are (derived) $p$-adic.

For the first proof, the Hodge--Tate comparison shows that the Hodge--Tate complex 
\[ (R \otimes_S S_\infty)_{\mathrm{HT},0} := \overline{\Prism}_{R \otimes_S S_\infty/A_{\inf}(S_\infty)}\] 
is $p$-completely flat over $R \otimes_S S_\infty$: this follows by considering the conjugate filtration (see \cite[Theorem 1.14 (1)]{BhattScholzepPrismaticCohomology}) and using that each $(\wedge^k L_{R\otimes_SS_\infty/S_\infty})^{\wedge p}[-k] \cong S_\infty\widehat{\otimes}_S(\wedge^k L_{R/S}[-k])$ has Tor dimension $\leq 0$ as $R$ is a complete intersection (see \cite[\href{https://stacks.math.columbia.edu/tag/08SH}{Tag 08SH}]{stacks-project}). Hence, the above complex has vanishing local cohomology for $i < d$ (as the same is true for $R\otimes_SS_\infty$). We then pass to direct limits over the Frobenius as in the proof of \autoref{prop:formally_etale_perfd}. More precisely, by definition of perfectoidization, $R^{S_\infty}_{\perfd}=(\varinjlim (R \otimes_S S_\infty)_{\mathrm{HT},n})^{\wedge p}$ where
$$(R \otimes_S S_\infty)_{\mathrm{HT},n} := \phi_*^n\Prism_{R \otimes_S S_\infty/A_{\inf}(S_\infty)} \otimes^{\myL}_{A_{\inf}(S_\infty)}(A_{\inf}(S_\infty)/\xi)$$
where $\xi$ denotes the distinguished element so that $A_{\inf}(S_\infty)/\xi=S_\infty$. It is enough to show that $(R \otimes_S S_\infty)_{\mathrm{HT},n}/\phi^{-n}(\xi)$ has vanishing local cohomology for $i<d-1$ (note that $\phi^{-n}(\xi)$ has a power in $(p,\xi)$ thus modulo $\xi$ it is contained in the radical of $(p)$). But after modulo $\phi^{-n}(\xi)$, we have 
$$(R \otimes_S S_\infty)_{\mathrm{HT},n}/\phi^{-n}(\xi) = F_*^n(\overline{\Prism}_{R \otimes_S S_\infty/A_{\inf}(S_\infty)}/\phi^n(\xi))=F_*^n((R \otimes_S S_\infty)_{\mathrm{HT},0}/\phi^n(\xi)).$$
Thus the $n=0$ case established above (with the harmless Frobenius pushforward) clearly implies the vanishing of local cohomology of the right hand side above for $i<d-1$. 

For the proof via Andr\'e's flatness lemma, note that  $R^{S_\infty}_\perfd = S_\infty /((f_1,\dots,f_c) S_{\infty})_\perfd$. By Andr\'{e}'s flatness lemma (\cite[Theorem 7.14]{BhattScholzepPrismaticCohomology}) again, we can perform a $p$-complete faithfully flat extension of perfectoid rings $S_\infty \to T$ such that $f_1,\dots,f_c$ have compatible system of $p$-power roots in $T$. Since $S_\infty$ is faithfully flat over $S$ we know that $T$ is $p$-complete faithfully flat over $S$ and thus honestly faithfully flat over $S$ (since $S$ is Noetherian). Let $y_1, y_2, \dots, y_d$ be a regular sequence on $R$. Thus $f_1,\dots,f_c, y_1, y_2,\dots, y_d$ is a regular sequence on $S$ and so remains a regular sequence on $T$. It follows that $y_1, y_2,\dots, y_d$ is a regular sequence on $T/(f_1^{1/p^e},\dots, f_c^{1/p^e})$ for all $e>0$ and thus $y_1, y_2,\dots, y_d$ is a regular sequence on $T/(f_1^{1/p^\infty}, \dots, f_c^{1/p^\infty})$. In particular, we know that $\myR\Gamma_\m(T/((f_1,\dots, f_c)T)_\perfd)\cong \myR\Gamma_\m(T/(f_1^{1/p^\infty},\dots, f_c^{1/p^\infty}))\in D^{\geq d}$ where the first isomorphism follows as $T/((f_1,\dots, f_c) T)_\perfd$ and $T/(f_1^{1/p^\infty}, \dots, f_c^{1/p^\infty})$ agree up to derived $p$-adic completion \cite[Lemma 2.3.2]{CaiLeeMaSchwedeTucker}. Since $S_\infty \to T$ is $p$-complete faithfully flat, $R^{S_\infty}_\perfd= S_\infty /(f_1,\dots,f_c)_\perfd\to T/((f_1,\dots, f_c)T)_\perfd$ is $p$-complete faithfully flat and thus $\myR\Gamma_\m(R^{S_\infty}_\perfd)\in D^{\geq d}$ as well. This proves that $H_\m^i(R^{S_\infty}_\perfd)=0$ for all $i<d$.

We next show that ${H^i_{\fram}(R_{\perfd})} = 0 $ for all $i<d$ and the equivalence of \claperfdinj and \absperfdinj simultaneously. We will need an intermediate object. Let $\cO_C/W(k^{1/p^\infty})$ be the $p$-completed ring of integers in a perfectoid extension $C/W(k^{1/p^\infty})[1/p]$ which is the $p$-completion of a totally ramified Galois extension with Galois group $\Gamma=\mathbf{Z}_p$. Let $R^{\O_C}_{\perfd}:= (R\widehat{\otimes}_{C_k}\O_C)_{\perfd}$.  Note that the following diagram 
\[
\xymatrix{
\big(\cO_C[x_1,\dots,x_n]\big)_\perfd \ar[r] \ar[d] & \cO_C[x^{1/p^\infty}_1,\dots,x^{1/p^\infty}_n]^{\wedge p} \ar[d] \\
\big(\cO_C\widehat{\otimes}_{C_k}C_k[[x_1,\dots,x_n]]\big)_\perfd \ar[r] & \big(\cO_C\widehat{\otimes}_{C_k}C_k[[x_1,\dots,x_n]]\big)[x^{1/p^\infty}_1,\dots,x^{1/p^\infty}_n]^{\wedge p}
}
\]
is a pushout (it is a pushout before we apply perfectoidization, and perfectoidization commutes with pushout, see \cite[Proposition 8.13]{BhattScholzepPrismaticCohomology}). Since the map in the top row is $p$-completely descendable by \cite[Lemma 8.6]{BhattScholzepPrismaticCohomology} (after going modulo the distinguished element),\footnote{Here $p$-completely descendable means after $-{\otimes}^\myL_{\mathbf{Z}_p}\mathbf{F}_p$, the map is descendable in the sense of \cite{MathewGaloisGroupStableHomotopy} (see \cite[Definition 11.14]{BhattScholzeProjectivityWitt}). All we need for descendability is that it is stable under base change, and that if $A\to B$ is descendable then $A\cong \varprojlim B^{\bullet/A}$, see \cite[Theorem 11.15]{BhattScholzeProjectivityWitt}.} 
so is the map in the bottom row, and thus by base change, the map $R^{\O_C}_{\perfd} \to R^{S_\infty}_{\perfd}$ is also $p$-completely descendable. In particular, we have 
$$
    R^{\O_C}_{\perfd} \cong \varprojlim {R^{S_\infty}_{\perfd}} ^{\bullet/R^{\O_C}_{\perfd}} = \varprojlim \left(R^{S_\infty}_{\perfd} \rightrightarrows R^{S_\infty}_{\perfd} \widehat{\otimes}^\mathbf{L}_{R^{\O_C}_{\perfd}} R^{S_\infty}_{\perfd} \, \substack{\longrightarrow\\[-1em] \longrightarrow \\[-1em] \longrightarrow} \cdots \right)
$$
and each face map in ${R^{S_\infty}_{\perfd}} ^{\bullet/R^{\O_C}_{\perfd}}$ is $p$-complete faithfully flat map of perfectoid rings by Andr\'{e}'s flatness lemma (\cite[Theorem 7.4]{BhattScholzepPrismaticCohomology}) since we have identifications 
$$R^{S_\infty}_{\perfd} \widehat{\otimes}^\mathbf{L}_{R^{\O_C}_{\perfd}} R^{S_\infty}_{\perfd} \cong (R^{S_\infty}_{\perfd} \widehat{\otimes}^\mathbf{L}_{(S\widehat{\otimes}_{C_k}\cO_C)} {S_\infty})_{\perfd}$$
by \cite[Proposition 8.13]{BhattScholzepPrismaticCohomology}.
Since $\myR\Gamma_\m(R^{S_\infty}_\perfd)\in D^{\geq d}$, it follows that $\myR\Gamma_\m(R^{\O_C}_\perfd)\in D^{\geq d}$ and that $H_\m^d(R^{\O_C}_\perfd)$ is the equalizer of $H_\m^d(R^{S_\infty}_{\perfd})$ $\rightrightarrows$ $H_\m^d$ $(R^{S_\infty}_{\perfd} \widehat{\otimes}^\mathbf{L}_{R^{\O_C}_{\perfd}} R^{S_\infty}_{\perfd})$. 
In particular, we know that $H_\m^d(R^{\O_C}_\perfd)\to H_\m^d(R^{S_\infty}_\perfd)$ is injective. 

Finally, write $\gamma \in \mathbf{Z}_p$ for a generator of the Galois group. By \autoref{prop:galois_pullback}, we have a pullback diagram 
\[
\xymatrix{
R_\perfd \ar[r] \ar[d] & \text{fib}(R^{\O_C}_\perfd \xrightarrow{\gamma -1}R^{\O_C}_\perfd) \ar[d] \\
(R/p)_\perf \ar[r] & (R/p)_\perf \oplus (R/p)_\perf[-1]
}.
\]
Applying $\myR\Gamma_\m(-)$ to the above diagram and noting that $\myR\Gamma_\m(R^{\O_C}_\perfd)\in D^{\geq d}$ and that \[ \myR\Gamma_\m((R^{\O_C}/p)_\perf)=\myR\Gamma_\m((R/p)_\perf)\cong H_\m^{d-1}((R/p)_\perf)[-(d-1)]
\]
since $R/p$ is Cohen-Macaulay, we obtain that $\myR\Gamma_\m(R_\perfd)\in D^{\geq d}$ (i.e., $H_\m^i(R_\perfd)=0$ for all $i<d$) and that
$H_\m^d(R_\perfd)$ is isomorphic to the kernel of the induced map 
\[
H_\m^d(R^{\O_C}_\perfd) \xrightarrow{\gamma -1} H_\m^d(R^{\O_C}_\perfd).
\]
In particular, we have that $H_\m^d(R_\perfd)\to H_\m^d(R^{\O_C}_\perfd)$ is injective. Putting these together, we see that $R$ is \claperfdinj if and only if $H_\m^d(R)\to H_\m^d(R_\perfd)$ is injective, that is, $R$ is \absperfdinj. 
\end{proof}

\begin{corollary}
\label{cor:LCIallequivalent}
Let $R$ be a Noetherian ring with $p$ in its Jacobson radical. Suppose $R$ is LCI (i.e., all local rings of $R$ are complete intersections). Then $R$ being perfectoid pure, lim-perfectoid pure, perfectoid injective, and lim-perfectoid injective are all equivalent.
\end{corollary}
\begin{proof}
By \autoref{lem:checklocalrings}, \autoref{lem:checklocalringsAbsPure}, \autoref{lem:pfdpurecompletion}, and \autoref{lem:abspurecompletion}, to show these notions are equivalent, we may assume that $R$ is a Noetherian complete local ring. The result then follows from \autoref{lem:pure_injective_comparison} and \autoref{thm: lci case}.
\end{proof}

It is natural to ask if one can weaken the LCI condition in the above.  In particular we expect the following:

\begin{conjecture}
    If $R$ is \absperfdpure (respectively \absperfdinj) then $R$ is \claperfdpure (respectively \claperfdinj).
\end{conjecture}

\subsection{Weak normality}
In this section, we show the weak normality of \claperfdinj rings (and even \absperfdinj rings).  We begin with some preliminaries.

The following fact is probably well-known to experts and we give proofs for completeness. For the definition and basic properties of seminormal and absolutely weakly normal rings, we refer to \cite[\href{https://stacks.math.columbia.edu/tag/0EUK}{Tag 0EUK}]{stacks-project} or \cite[Appendix B]{RydhSubmersionsEffectiveDescent}.

\begin{lemma}
\label{lem.perfectoidweaklynormal}
Any perfectoid ring $A$ is absolutely weakly normal.
\end{lemma}
\begin{proof}
Let $A\to B$ be a weakly subintegral extension (i.e., $\Spec(B)\to \Spec(A)$ is a universal homeomorphism). By \cite[\href{https://stacks.math.columbia.edu/tag/0EUR}{Tag 0EUR}]{stacks-project}, it is enough to show that $B$ admits a unique map to $A$. By \cite[\href{https://stacks.math.columbia.edu/tag/0EUJ}{Tag 0EUJ}]{stacks-project}, we may write that $B$ as a filtered colimit of $B_j$ such that each $B_j$ is finitely presented and finite over $A$ (in particular, $B$ is derived $p$-complete) and $\Spec(B_j)\to \Spec(A)$ is a universal homeomorphism. 

Now $A[1/p]$ is a perfectoid Tate ring and thus seminormal by \cite[Theorem 3.7.4]{KedlayaLiuRelativepadicHodgeII}. Since $A[1/p]\to B_j[1/p]$ is a subintegral extension (as $A[1/p], B_j[1/p]$ contain $\mathbf{Q}$),  we have $A[1/p]\cong B_j[1/p]$. But then by \cite[Corollary 8.12]{BhattScholzepPrismaticCohomology}, the following diagram
\[\xymatrix{
A \ar[r] \ar[d] & (B_j)_\perfd \ar[d] \\
(A/p)_\perf \ar[r] & (B_j/p)_\perf \\ 
}
\]
is a pullback square. Since $\Spec(B_j/p)\to \Spec(A/p)$ is a universal homeomorphism of rings of characteristic $p$ we have $(A/p)_\perf\cong (B_j/p)_\perf$. It follows from the pullback diagram that $A\cong (B_j)_\perfd$. This means each $B_j$ admits a unique map to $A$ and thus $B$ admits a unique map to $A$. 
\end{proof}

\begin{lemma}
\label{lem.weaklynormal}
    Suppose $(R,\m)$ is a Noetherian local reduced ring such that $R_P$ is weakly normal for all prime ideals $P\neq \m$. Let $R^{\wn}$ be the weak normalization of $R$. If $H^1_{\m}(R) \to H^1_\m(R^{\wn})$ is injective, then $R\cong R^{\wn}$, i.e., $R$ is weakly normal.
\end{lemma}
\begin{proof}    
    Let $X$ denote the punctured spectrum of $R$ and consider the following diagram which has exact rows by \cite[Tag 0DWR]{stacks-project}:
    \[
        \xymatrix{
            0 \ar[r] & R \ar[d] \ar[r] & \Gamma(X, \cO_{X})\ar@{=}[d] \ar[r] & H^1_\m(R) \ar@{^{(}->}[d] \ar[r] & 0\\
            0 \ar[r] & R^{\wn}  \ar[r] & \Gamma(X, \cO_{X^{\wn}})  \ar[r] & H^1_\m(R^{\wn}) \ar[r] & 0            
        }. 
    \]
Chasing the diagram we find that $\Gamma(X, \cO_X)$ $\to H^1_{\m}(R^{\wn})$ is surjective, thus $H^1_\m(R) \to H^1_\m(R^{\wn})$ is also surjective and hence it is an isomorphism. It follows that $R \to R^{\wn}$ is an isomorphism as well. 
\end{proof}

\begin{corollary}
    \label{cor.ArcInjectiveImpliesWeaklyNormal}
   Let $(R,\fram)$ be a Noetherian ring with $p$ in its Jacobson radical.  If $R$ is \absperfdinj then $R$ is reduced and weakly normal.
\end{corollary}
\begin{proof}
Without loss of generality, we may assume that $R$ is complete local with maximal ideal $\fram$ by \autoref{lem.absPerfdPureLocalizes} and \autoref{lem:abspurecompletion} (and note that $\widehat{R}$ is reduced and weakly normal implies $R$ is so by \cite[Corollary II.3]{ManaresiSomePropertyWeaklyNormal}). In particular, in what follows we will assume $R$ admits a dualizing complex.

Since $R\to R_{\perfd}$ factors through $R_{\red}$, the condition implies that $H_\m^i(R)\to H_\m^i(R_{\red})$ is injective for all $i$. By \autoref{lem:DualizingComplex}, we know that $H^0_P(R_P)\to H^0_P((R_P)_{\red})$ is injective for all $P\in\Spec(R)$.  The latter is a field if $P$ is a minimal prime and zero otherwise. It follows that $R$ is satisfies $(\text{R})_0$ and $(\text{S})_1$ so $R$ is reduced.


We next show that $R$ is weakly normal, we already know that $R$ is reduced. Since any perfectoid ring is absolutely weakly normal by \autoref{lem.perfectoidweaklynormal}, any map $R\to B$ factors canonically through $R^{\wn}$. Therefore we also get a factorization $R\to R^{\wn}\to R_{\perfd}$ and thus the condition implies that $H_\m^i(R)\to H_\m^i(R^{\wn})$ is injective for all $i$. By \autoref{lem:DualizingComplex}, we know that $H_P^i(R_P) \to H^i_P((R_P)^{\wn})$ is injective for all $i$ as weak normalization commutes with localization. Now suppose $R$ is not weakly normal, we fix $Q \in \Spec (R)$ of minimal height such that $R_Q$ is not weakly normal. Then $R_Q$ is weakly normal on the punctured spectrum and $H_Q^1(R_Q) \to H^1_Q((R_Q)^{\wn})$ is injective. By \autoref{lem.weaklynormal}, $R_Q$ is weakly normal and we arrive at a contradiction.
\end{proof}

\section{Comparison with log canonical and Du Bois singularities}

Our goal in this section is to study partial analogs of the main results of \cite{HaraWatanabeFRegFPure} in mixed characteristic (also see \cite{MillerSchwedeSLCvFP,SchwedeFInjectiveAreDuBois,BhattSchwedeTakagiweakordinaryconjectureandFsingularity}) at least under certain index assumptions.
We begin by considering the Du Bois cases.

\begin{proposition}
    \label{prop.ArcImpliesR1overpIsDuBois}
    Suppose $(R, \fram)$ is a \absperfdinj \ 
 Noetherian local ring of mixed characteristic $(0, p> 0)$, essentially of finite type over a mixed characteristic DVR.  Then $R[1/p]$ has Du Bois singularities.
\end{proposition}

\begin{proof}
Let $X=\Spec(R)$ and $\widehat{X}=\Spf(R)$ be the $p$-adic completion.  Let $\pi:\widehat{X}_{\mathrm{top}}\to {X}_{\mathrm{top}}$ be the morphism of sites given by the $p$-completion functor
$\pi:\mathrm{Sch}/X\to\mathrm{FSch}/\widehat{X}$, where $\mathrm{top}$ denotes either the Zariski or arc-topology.  Then pushforward by $\pi$ gives the exact functor $\pi_*: \mathrm{Shv}(\widehat{X})\to \mathrm{Shv}({X})$ fitting in the diagram with arc to Zariski pushforwards:  
\[
\xymatrix{
    \cO_X\ar[r]\ar[d] &\myR\Gamma_{\arc}\cO_X^{\arc}\ar[d]\\
    \pi_*\cO_{\widehat{X}}\ar[r] & \pi_*\myR\Gamma_{\arc}\cO_{\widehat{X}}^{\arc}.}
    \]
The vertical arrows come from the natural maps $\cO_X\to\pi_*\pi^*\cO_X=\pi_*\cO_{\widehat{X}}$ and similarly for $\cO_X^{\arc}$.  The existence of the right vertical arrow is explained by noting that both $\pi_*$ and $\myR\Gamma_{\arc}$ are derived pushforwards by maps of sites, and hence commute.   
The left vertical arrow is an isomorphism after applying $H^i_{\mathfrak{m}}(-)$, and the bottom row is injective after applying $H^i_{\mathfrak{m}}(-)$ because $R$ is lim-perfectoid injective (since $\myR\Gamma_{\arc}\cO_{\widehat{X}}^{\arc}$ identifies with $R_\perfd$ by \autoref{prop:h-sheafification}).  Thus the top row is also injective after applying $H^i_{\mathfrak{m}}(-)$. Therefore, by local duality, localization, and local duality again (see \autoref{lem:DualizingComplex}), the natural map 
$$H^i_Q(\cO_{X,Q})\to H^i_Q((\myR\Gamma_{\arc}\cO^{\arc}_{X})_Q) \cong H^i_Q((\DuBois{X[1/p]})_Q)$$ 
is injective for all $Q\in X[1/p]$ which implies that $X[1/p]$ is Du Bois {(cf.\ \cite[Lemma 2.2]{KovacsDuBoisLC2})}. {Here the final isomorphism follows from the fact that $(\DuBois{X[1/p]})_Q$ is equal to $(\myR\Gamma_{h}\cO^{h}_{X})_Q$ (\cite[Theorem 4.13]{LeeLocalAcyclicFibrationsAndDeRham} or \cite[Proposition 6.10]{HuberJorderDifferentialFormsHTopology}}), and that the $h$ and arc-topologies agree in the Noetherian case (\cite[Proposition 2.6]{BhattMathewArc} and \cite[Section 2]{BhattScholzeProjectivityWitt}).
\end{proof}

\begin{remark}
    We expect that the hypothesis that $R$ is essentially of finite type over a mixed characteristic DVR can be replaced by the assumption that $R$ is excellent and has a dualizing complex, see \cite{MurayamaInjectivityAndCubicalDescentForSchemesStacksSpaces}.  The missing piece is that we do not know a reference that $H^i_Q(\cO_{X,Q}) \to H^i_Q((\DuBois{X[1/p]})_Q)$ surjects in that generality, see \cite{KovacsSchwedeDBDeforms}.  Since we are proving it injects, the surjection implies it is an isomorphism by duality.
\end{remark}

We now move into log canonical singularities.  First we need a lemma.

\begin{lemma}\label{lem.BlowupSquareSequence}
Suppose $(R,\fram)$ is a $p$-complete Noetherian local ring and $\pi:Y\to X=\Spec R$ is a proper map which is an isomorphism over the complement of $Z\subseteq X$ and set $E=\pi^{-1}(Z)_{\red}$.  Let $C$ be the following pullback in ${D}(R)$
\[
\xymatrix{
 C\ar[r]\ar[d] & \myR\Gamma(Y,\cO_Y)\ar[d]\\
\myR\Gamma(Z,\cO_Z)\ar[r] & \myR\Gamma(E,\cO_E).
}
\]
Then we have a factorization
\[
      R \to C \to R_{\perfd}.
    \]
\end{lemma}
\begin{proof}
Since $R$ is Noetherian $p$-complete and $\myR\Gamma(Y,\cO_Y)$, $\myR\Gamma(X,\cO_Z)$ and $\myR\Gamma(Y,\cO_E)$ are all in $D^b_{\coherent}(R)$, they are all $p$-complete, as is $C$.
Now since  $R_{\perfd}$ is an arc sheaf on $\Spf(R)$ (see \autoref{prop:h-sheafification}) we have another pullback diagram in $\widehat{D}(R)$:
\[
\xymatrix{
 R_{\perfd}\ar[r]\ar[d] & \myR\Gamma(\widehat{Y},\cO_{\widehat{Y},\perfd})\ar[d]\\
\myR\Gamma(\widehat{Z},\cO_{\widehat{Z},\perfd})\ar[r] & \myR\Gamma(\widehat{E},\cO_{\widehat{E},\perfd}),
}
\]
where $\widehat{(-)}$ denotes the $p$-adic formal completion of the corresponding scheme. Hence the required factorization follows immediately from the universal property enjoyed by the latter pullback. 
\end{proof}

\begin{proposition}[{\cf \cite{KovacsSchwedeSmithLCImpliesDuBois}}]
    \label{prop.KSSPushForwardNormalCase}
    Suppose $(R, \fram)$ is a Noetherian local equidimensional ring with a dualizing complex.  Suppose $R$ is \absperfdinj and we are given a proper birational map $\pi : Y \to X = \Spec R$ satisfying the following.
    \begin{itemize}
        \item{} $\pi$ is an isomorphism outside a subset of codimension $\geq 2$ on $X$.
        \item{} The reduced exceptional set $E \subseteq Y$ is pure codimension $1$ and $Y$ is regular at the minimal primes of $E$.  
    \end{itemize}
    We then have that 
    \[
        \pi_* \omega_Y(E) = \omega_X.
    \]
    In particular, if $R$ is {normal} and quasi-Gorenstein, then $R$ is log canonical.
\end{proposition}
\begin{proof}
We may assume that $R$ is $p$-complete, and form $C$ as in \autoref{lem.BlowupSquareSequence}.  From \autoref{lem.BlowupSquareSequence}, and the hypothesis that $R$ is \absperfdinj, we see that $R \to C^{\mydot}$ is injective on local cohomology.  Suppose $Z = V(I)$ for some $I \subseteq R$.

\begin{claim}
    \label{clm.CDotDualizes}
    Suppose $d = \dim R$.  Then $\myH^{-d}({\bf D} (C)) = \Gamma(Y, \omega_Y(E)) $ where ${\bf D}$ denotes Grothendieck duality $\myR\Hom_R(-, \omega_R^{\mydot})$.
\end{claim}
\begin{proof}[Proof of claim]
    Letting $D$ be the fiber of the diagram defining $C$, we have
    the following diagram in which both rows are fiber sequences:
    \[
        \xymatrix{
            D \ar[r] \ar@{=}[d] & C  \ar[d] \ar[r] & R/I \ar[d]  \\
            D \ar[r] & \myR \Gamma(Y, \cO_Y) \ar[r] & \myR \Gamma(E, \cO_E)  
        }
    \]
    Using the bottom row we see that $D = \myR \Gamma(Y,\cO_Y(-E))$.
    We apply Grothendieck duality ${\bf D}$ to this diagram to obtain:
    \[
        \xymatrix{
            {\bf D} (D) & {\bf D} (C) \ar[l] & \ar[l] \omega_{R/I}^{\mydot}  \\
           \ar@{=}[u]  {\bf D} (D) & \ar[u]  \myR \Gamma(Y, \omega_Y^{\mydot}) \ar[l] & \ar[u]  \ar[l]\myR \Gamma(E, \omega_E^{\mydot}) 
        }
    \]
    Since the codimension of $Z$ is at least 2, we see that $\myH^{-d}({\bf{D}}( {D})) \cong \myH^{-d}({\bf{D}}( {C}))$ from the first row.  On the other hand, since $D = \myR \Gamma(Y,\cO_Y(-E))$, we have that ${\bf D}(D) = \myR\Gamma(Y, \omega_Y(E))$. This proves the claim.
\end{proof}
Since $H^d_{\fram}(R) \to H^d_{\fram}(\myR \Gamma(Y, C))$ injects, the dual map $\pi_* \omega_Y(E) \to \omega_X$ surjects.  One verifies that generically this map is the identity ($\pi$ is birational) and so the map is also injective.  This completes the proof. 
\end{proof}

\begin{remark}
    These results can also be obtained completely analogously without using the infinity category framework, but then the octahedral axiom must be invoked.
\end{remark}

\subsection{Cyclic covers}

Cyclic covers are a classical way to study singularities in characteristic zero or $p > 0$.  For instance, unramified cyclic covers provide a convenient way to generalize some results from the quasi-Gorenstein to the $\bQ$-Gorenstein setting.  In this section we explore the behavior of \classicalperfdpure singularities under cyclic covers.  We unfortunately must restrict ourselves to index-not-divisible-by-$p$ case.
    We begin with a lemma that is well known to experts but for which we do not know a statement in our generality.

    \begin{lemma}
        \label{lem.ImagedContainedInS2}
        Suppose $(R, \m )$ is an S2 Noetherian local ring and $M$ is an $R$-module with a given map $\phi : M \to K(R)$.  Suppose further that there exists an ideal $J = (f_1, \dots, f_n)$ of codimension $\geq 2$ so that the induced maps
        \[
            \phi(M_{f_i}) \subseteq R_{f_i}.
        \]
        Then $\phi(M) \subseteq R$.
    \end{lemma}
    \begin{proof}
        This may be checked on the finitely generated submodules of $M$, where the statement is well known, see for instance \cite{HartshorneGeneralizedDivisorsOnGorensteinSchemes}.
    \end{proof}

We are ready to prove our first result of this sort.

\begin{proposition}
    \label{prop.CyclicCoverNIndexStuff}
    Suppose $(R, \m)$ is a G1 and S2 reduced complete Noetherian local ring.  Suppose that $D$ is a divisor\footnote{In particular, $D$ is the principal divisor associated to a non-zero divisor at every height one prime of $R$.  In other words, $D$ is an \emph{almost-Cartier} divisor in the sense of \cite{HartshorneGeneralizedDivisorsOnGorensteinSchemes}.}  on $\Spec R$ in the sense of \cite{KollarKovacsSingularitiesBook}.    Suppose additionally that $nD \sim 0$ for some $n > 0$, where $n$ is the smallest positive integer with this property.
    
    Let $(R, \m) \subseteq (S, \n)$ be an associated cyclic cover with $S = R \oplus R(-D) \oplus \dots \oplus R(-(n-1)D)$ (depending on a choice of isomorphism $R(nD) \cong R$).   Choose $A = C_k\llbracket x_2, \dots, x_n\rrbracket \to R$ a map making $R$ into a finite $A$-module (for instance, this might be a surjection, or a Noether-Cohen normalization).  We fix $A_{\infty}$ as in \autoref{eq.SInftyDefinition}.

    \begin{enumerate}
        \item  If $R$ is \classicalperfdpure, then $R[1/p] \to S^{A_\infty}_{\perfd}[1/p]$ splits. \label{prop.CyclicCoverNIndexStuff.NoIndexAssumption}
        \item  If $R$ is \classicalperfdpure and $p \nmid n$, then $R \to S^{A_{\infty}}_{\perfd}$ splits.  \label{prop.CyclicCoverNIndexStuff.IndexAssumption}
        \item  If $R$ is \classicalperfdpure and $p \nmid n$, then $S$ is \classicalperfdpure. \label{prop.CyclicCoverNIndexStuff.PerfdPureS}
    \end{enumerate}
\end{proposition}
\begin{proof}
    Notice that $K := K(R) = K_1 \times \dots \times K_t$ is a product of fields, and over each one $K_i$, $S \otimes_{R} K_i$ is an $n$-dimensional $K_i$ vector space.  In fact, by construction, over the complement of a codimension $\geq 2$ closed subset of $\Spec R$, $S$ is locally free of rank $n$ over $R$.  Since $R$ is S2, it follows we have a trace map $T : S \to R$ sending $1 \mapsto n$.  

    The key tool is that  $R[1/p] \subseteq S[1/p]$ is \emph{\'etale in codimension 1} (that is: \emph{quasi-\'etale})  and if $p \nmid n$, then $R \subseteq S$ is \emph{quasi-\'etale}.  This implies that the trace map $T[1/p] : S[1/p] \to R[1/p]$ generates $\Hom_{R[1/p]}(S[1/p], R[1/p])$ as an $S[1/p]$-module in general and that $T$ generates $\Hom_R(S, R)$ as an $S$-module if $p$ does not divide $n$.
    We form the following diagram.
    \[
        \xymatrix{
            R \ar[r] \ar[d] & S \ar[d] \\
            \ar@/^2pc/[u]^{\psi} R^{A_{\infty}}_{\perfd} \ar[r] & S^{A_{\infty}}_{\perfd}
        }
    \]
    where $\psi$ is the splitting that comes from the fact that $R$ is perfectoid-pure.  

    Note that over the locus where $R \subseteq S$ is \'etale, we have that $(R^{A_{\infty}}_{\perfd} \otimes_R S) \to (R^{A_{\infty}}_{\perfd} \otimes_R S)_{\perfd} = S^{A_{\infty}}_{\perfd}$ is an isomorphism \cite[Theorem 10.9]{BhattScholzepPrismaticCohomology}.  In particular, $(R^{A_{\infty}}_{\perfd} \otimes_R S) \to (R^{A_{\infty}}_{\perfd} \otimes_R S)_{\perfd}$ is an isomorphism over the generic points $\Spec K_i$ of characteristic $0$, 
    ie, those that make up $K[1/p]$ (which equals $K$ as long as $R$ has no minimal prime containing $p$).  Tensoring with $K[1/p]$, since $K[1/p] \subseteq K(S)[1/p]$ is finite \'etale, we have that $K[1/p] \otimes_R S^{A_{\infty}}_{\perfd} = K[1/p] \otimes_R S \otimes_R R^{A_{\infty}}_{\perfd}$
    hence we obtain a map 
    \[
        T_{\perfd}' : K[1/p] \otimes_R S^{A_{\infty}}_{\perfd} \to K[1/p] \otimes_R R^{A_{\infty}}_{\perfd} 
    \]
    induced by trace.  Composing, we obtain:
    \[
        \phi : S^{A_{\infty}}_{\perfd} \to K[1/p] \otimes S^{A_{\infty}}_{\perfd} \xrightarrow{T_{\perfd}'} K[1/p] \otimes R^{A_{\infty}}_{\perfd} \xrightarrow{K[1/p] \otimes \psi} K[1/p].
    \]
    For the first statement, it suffices to show that this map lands in $R[1/p]$.  Indeed, because it sends $1 \mapsto n$ which is a unit in $R[1/p]$, it will be surjective onto $R[1/p]$ if its image is contained in $R[1/p]$.
    
    For any $f \in R$ such that $R[1/(fp)] \to S[1/(fp)]$ is \'etale, 
    \[
        (S^{A_{\infty}}_{\perfd})[1/(fp)] \cong (R^{A_{\infty}}_{\perfd})[1/(fp)] \otimes_{R[1/(fp)]} S[1/(fp)]
    \]
    again by \cite[Theorem 10.9]{BhattScholzepPrismaticCohomology}.  
    Hence, arguing as above: $T_{\perfd}' \Big ( S^{A_{\infty}}_{\perfd}[1/(fp)]\Big) \subseteq (R^{A_{\infty}}_{\perfd})[1/(fp)]$.
    Thus $\phi(S^{A_{\infty}}_{\perfd}) \subseteq R[1/(fp)]$ for all such $f$.  Therefore, by \autoref{lem.ImagedContainedInS2},
    \[ 
        \phi(S^{A_{\infty}}_{\perfd}) \subseteq  R 
    \]
    This proves \autoref{prop.CyclicCoverNIndexStuff.NoIndexAssumption}. 

    The proof of \autoref{prop.CyclicCoverNIndexStuff.IndexAssumption} is essentially the same.  Run the same argument without inverting $p$ to obtain $\phi : S^{A_{\infty}}_{\perfd} \to K$. Then argue that the image lands in $R$ exactly as before by verifying it after inverting $f \in R$ such that $R[1/f] \subseteq S[1/f]$ is \'etale.

    For \autoref{prop.CyclicCoverNIndexStuff.PerfdPureS}, we follow \cite{CarvajalRojasFiniteTorsors}.  Since $R \subseteq S$ is quasi-\'etale and $S$ is S2, we see that $\Hom_R(S, R) \cong S$ (generated as an $S$-module by the trace map).  Thus by Hom-tensor adjointness, $\phi : S_{\perfd}^{A_{\infty}} \to R$ factors as 
    \[ 
        S_{\perfd}^{A_{\infty}} \xrightarrow{\phi_S} S \xrightarrow{\Tr} R.
    \]
    If $\phi_S$ was not surjective, its image would lie in $\n = \m + S_{\geq 1}$, the unique maximal ideal of $S$.  However, since $\Tr(\n) \subseteq \m$, and the composition $\phi$ surjects, this is impossible.  This completes the proof.
\end{proof}

We did not really need the above extension $R \subseteq S$ to be a cyclic cover.  A variant of our argument above also works in the following situation.  Note in this case we do not assume that $p \not| [K(S) : K(R)]$

\begin{corollary}
    Suppose $(R, \fram) \subseteq (S, \frn)$ is a split quasi-\'etale extension of complete local Noetherian normal domains with $R$ perfectoid-pure.  Then $S$ is also \claperfdpure.
\end{corollary}
\begin{proof}
The trace map $T : S \to R$ is surjective as the splitting must be a multiple of the generator $T \in \Hom_R(S, R)$.  Pick $x \in S$ with $T(x) = 1$.  The induced map 
    \[
        \phi : S^{A_{\infty}}_{\perfd} \to K \otimes_R S^{A_{\infty}}_{\perfd} = K \otimes_R S \otimes_R R^{A_{\infty}}_{\perfd} \xrightarrow{T'_{\perfd}} K\otimes_R R^{A_{\infty}}_{\perfd} \xrightarrow{K \otimes \psi} K
    \]
    then sends the image of $x$ in $S^{A_{\infty}}_{\perfd}$ to $1$.  On the other hand, mimicking the proof of \autoref{prop.CyclicCoverNIndexStuff} \autoref{prop.CyclicCoverNIndexStuff.IndexAssumption}, we see that the image of $\phi$ is contained in $R$ as $R$ is S2.  Hence $\phi : S^{A_{\infty}}_{\perfd} \to R$ is surjective.
    Repeating the argument of \autoref{prop.CyclicCoverNIndexStuff} \autoref{prop.CyclicCoverNIndexStuff.PerfdPureS} then proves that $S$ is also \claperfdpure.  
\end{proof}

The above corollary is in many ways more general than the cyclic cover statement.  However, we want the flexibility to handle the cyclic covers when $\Spec R$ has multiple irreducible components as we want to show that such $R$ are semi-log canonical (SLC) below.

We expect the results above to generalize to the case of general index $[K(S):K(R)]$ in the following way.  See \cite{CarvajalRojasFiniteTorsors} for the analog in characteristic $p > 0$.

\begin{conjecture}
    \label{conj.GeneralQuasiTorsorPerfdPure}
    Suppose that $R \subseteq S$ is a finite $\mu_n$-quasi-torsor\footnote{That is, there exists an open subset $U \subseteq \Spec R$ whose complement has codimension $\geq 2$, such that $\Spec S \to \Spec R$ is $\mu_n$-torsor.} over a Noetherian local reduced G1 and S2 ring $(R, \fram)$ of mixed characteristic.  If $R \to S$ is split, and $R$ is \claperfdpure, then so is $S$.
\end{conjecture}

One could also ask that  \autoref{prop.CyclicCoverNIndexStuff} and the above conjecture hold for \absperfdpure singularities.  On the other hand, based on the characteristic zero and characteristic $p > 0$ pictures, we do not expect \autoref{conj.GeneralQuasiTorsorPerfdPure} to hold for perfectoid-injective or \absperfdinj singularities.

\subsection{Log canonical singularities}

We now apply our work to log canonical singularities.

\begin{corollary}
    \label{cor.PerfdPureImpliesLC}
    Suppose $R$ is \claperfdpure, normal, and $\bQ$-Gorenstein of index not divisible by $p > 0$.  Then $R$ is log canonical.
\end{corollary}
\begin{proof}
    By \autoref{prop.CyclicCoverNIndexStuff}, a quasi-Gorenstein cyclic cover of index prime-to-$p$ is \claperfdpure and it is well known it is normal as the extension is quasi-\'etale.  Hence the cyclic cover $S$ is log canonical by \autoref{prop.KSSPushForwardNormalCase}.  Furthermore, since $K(R) \subseteq K(S)$ is Galois of index not divisible by $p$, we see that each divisorial valuation of $K(S)$ is tame over its restriction to $K(R)$, see for instance \cite[\href{https://stacks.math.columbia.edu/tag/09EA}{Tag 09EA}]{stacks-project} or \cite{KerzSchmidtOnDifferentNotionsOfTameness}.  Hence the usual computation of discrepancies holds (\cite[Proposition 5.20]{KollarMori}) and $R$ is also log canonical. 
\end{proof}

We expect that the hypothesis that the index is not divisible by $p > 0$ can be removed and perhaps also that \claperfdpure can be weakened to \absperfdpure.

\begin{conjecture}
    If $R$ is \absperfdpure, normal, and $\bQ$-Gorenstein, then $R$ is log canonical.
\end{conjecture}

One could generalize this to pairs as well, although we do not fully develop the theory of pairs in this paper (see \autoref{def.PerfdPurePair} for a first definition).

\begin{proposition}
    \label{prop.PerfdPureImpliesR1overpIsLC}
    Suppose that $R$ is \claperfdpure, normal, and $\bQ$-Gorenstein.  Then $R[1/p]$ is log canonical.
\end{proposition}
\begin{proof}
    Without loss of generality, we may assume that $R$ is complete local (as we can check whether $R$ is log canonical on a log resolution of the characteristic zero scheme $\Spec R[1/p]$).
    Again we have a cyclic cover $S = \bigoplus_{i = 0}^{n-1} R(-iK_R)$ where $n$ is the index of $K_R$, and where multiplication on $S$ is defined using an isomorphism $\omega_R^{(n)} \cong R$.  We have a map $\Phi : S^{A_{\infty}}_{\perfd} \to R$ sending $1 \mapsto n = [K(S) : K(R)]$ by the proof of \autoref{prop.CyclicCoverNIndexStuff}.  

    As $R[1/p] \subseteq S[1/p]$ is quasi-\'etale, it suffices to show that $S[1/p]$ is log canonical.
    
    Now, for any $C$ coming from $Y \to \Spec S$ as in \autoref{lem.BlowupSquareSequence}, 
    we have a factorization 
    \[
        R[1/p] \to S[1/p] \to C[1/p] \to (S^{A_{\infty}}_{\perfd})[1/p].
    \]
    The map $\Phi' := (1/n) \cdot \Phi[1/p]$ splits this inclusion.  As $R[1/p] \subseteq S[1/p]$ is quasi-\'etale, $\Tr$ generates $\Hom_{R[1/p]}(S[1/p], R[1/p])$ as an $S[1/p]$-module, it follows from $\Hom$-tensor adjointness that we can factor $\Phi'$ as 
    \[
        \Phi' : (S^{A_{\infty}}_{\perfd})[1/p] \xrightarrow{\Psi} S[1/p] \xrightarrow{\Tr} R[1/p]
    \]
    for some $S[1/p]$-linear $\Psi$.  Note $\Tr(\Psi(1)) = 1$.  

    Pick $Q \in \Spec R[1/p]$.
    As we are in characteristic zero, it suffices to show that $S[1/p]$ is log canonical at at least one prime $Q' \in \Spec S[1/p]$ lying over $Q$ (indeed, since $S[1/p]$ is generically Galois over $R[1/p]$ if it is log canonical at one $Q'$, it is at all $Q'$).  As $\Tr$ sends $\sqrt{QS[1/p]}$ into $Q$, it follows that $\Psi(1) \notin \sqrt{QS[1/p]}$ and hence $\Psi(1) \notin Q'$ for at least one $Q'$ lying over $Q$.  Localizing at $Q'$, we have that the composition  
    \[
        S_{Q'} \to C_{Q'} \to  (S^{A_{\infty}}_{\perfd})[1/p]_{Q'} \to S_{Q'}
    \]
    is an isomorphism.  Thus $S_{Q'} \to C_{Q'}$ is split, and so the Grothendieck dual
    \[
        S_{Q'} \cong \omega_{S_{Q'}} = \myH^{-d} \omega_{S_{Q'}}^{\mydot} \leftarrow \myH^{-d} {\bf D}(C_{Q'})
    \]
    is surjective where $d = \dim S_{Q'}$.  But by \autoref{clm.CDotDualizes}, $\myH^{-d} {\bf D}(C_{Q'}) = \Gamma(Y, \cO_Y(K_Y + E))_{Q'}$ where $E$ is the reduced exceptional divisor.  It follows that $S_{Q'}$ is log canonical, and the proof is complete.
\end{proof}

\subsection{Generalizations outside of the normal case}

The goal of the next section is to generalize the work done previously in this section outside of the normal case.  Beyond simply generalizing to the case of semi-log canonical singularities, such considerations are also necessary even if one assumes \autoref{conj.GeneralQuasiTorsorPerfdPure}.  Indeed, we expect that a $\mu_p$-quasi-torsor over a normal \claperfdpure singularity need not be normal.

Before we continue, we need a slightly nonstandard statement of a well known result attributed to Zariski and Abhyankar.  

\begin{theorem}[{\cf \cite[Section 5]{ArtinNeronModels}, \cite{ZariskiTheReductionOfTheSingsOfASurface, AbhyankarOnTheValuationsCentered}}]
    \label{thm.ArtinZariskiAbhyankar}
    Suppose $(R, \fram)$ is an excellent reduced Noetherian local ring with total ring of fractions $K(R) = K_1 \times \dots \times K_n$ and $v$ is a divisorial valuation over $X = \Spec R$ in some $K_i$.   Then by repeatedly blowing up the center of $v$ in $X$, we obtain a scheme $f : X' \to X$ such that over the irreducible component $X_i$ corresponding to $K_i$, we have that the valuation ring of $v$ is a stalk on $X_i'$ (the strict transform of $X_i$).  

    Furthermore, let $V \subseteq X$ denote the center of $v$ 
    and $\kappa : U \subseteq X$ be an open set contained in $X \setminus V$.  Let $X'' = \sheafspec \sA$ where $\sA$ is the normalization of $\cO_{X'} \subseteq \kappa_* \cO_{f^{-1}(U)}$ and let $X_j''$ denote the irreducible components corresponding to the $K_j$.  Then the composition $g : X'' \to X$ satisfies the following.
    \begin{enumerate}
        \item $g$ is an isomorphism over $U$.
        \item If $Z \subseteq X''$ is the center of $v$ on $X''$, then $X''$ is normal at the generic point of $Z$.
        \item More generally, at every height one point $\mu \in X_i''$ which is the generic point of an irreducible component of $X''_i \setminus g^{-1}(U)$, we have that $\cO_{X'', \mu}$ is a DVR.
    \end{enumerate}
\end{theorem}

\begin{proof}
    If $R$ is a domain, the first part of the statement (before ``Furthermore,") is well known and can be found in the cited reference.  

    Now, let us consider what happens if we run the algorithm on a non-irreducible $X$.  Note that if at each step, we consider the strict transform of $X_i$, this behaves exactly as the classical integral domain case.  Hence, ignoring the components $X_j$ for $j \neq i$, we now have a scheme $X'$ where one component $X_i'$ has a stalk $\cO_{X_i', \eta}$, at some point $\eta$, equal to the valuation ring of $v$.  
    We write 
    \[ 
        Y' = \bigcup _{j \neq i} X_j'
    \]  
    to be the union of irreducible components of $X'$ distinct from $X_i'$.

    Pick $\mu$ a height-1 point of $X_i'$ which maps into $V$, that is $\mu \in X'_i \setminus f^{-1}(U)$ (for example, $\mu = \eta$).  Note $\mu$ may also be a point of other $X'_{k}$ as well.  Consider $\kappa : f^{-1}(U) \to X'$ the inclusion, then $(\kappa_* \cO_{f^{-1}(U)})_{\mu}$ is the kernel of some 
    \[
        \prod_a \cO_{X',\mu}[h_a^{-1}] \to \prod_{a < b} \cO_{X',\mu}[h_a^{-1}, h_b^{-1}] 
    \]
    for some finitely many $h_a$'s defining the complement of $f^{-1}(U)$ in  $\Spec \cO_{X',\mu}$.  However, every $h_a \in \fram_\mu$, and so when we invert $h_a$, at least on the $X_i'$ component we obtain a field as $\mu$ has height one on $X_i$.  Hence each $\cO_{X',\mu}[h_a^{-1}] = K(X_i') \times \prod \cO_{Y',\mu}[h_a^{-1}]$.  It follows that 
    \[
        (\kappa_* \cO_{f^{-1}(U)})_{\mu} = K(X_i') \times (\kappa_* \cO_{Y' \cap f^{-1}(U)})_{\mu}.
    \]
    We then see that $\sA$, the normalization of $\cO_{X'} \subseteq \kappa_* \cO_{f^{-1}(U)}$ satisfies 
    \[
        \sA_{\mu} := \cO_{X_i, \mu} \times \sB_{\mu}
    \]
    where $\sB$ is the normalization of $\cO_{Y'}$ in $\kappa_* \cO_{f^{-1}(U) \cap Y'}$.   

    Set $X'' := \sheafspec \sA$.  Since integral closure commutes with localization, $X''$ is a DVR at the pre-images of $\mu$ in $X_j$.  We also call these points $\mu$.
    The composition $X'' \to X' \to \Spec R$ is a map with the desired properties as $X'' \to X'$ is an isomorphism over $f^{-1}U$.
\end{proof}

Birational maps on a finite cover of $X$ can be used to show that $X$ has (semi-)log canonical singularities.  In characteristic zero, this easily follows from log discrepancy formulas, but due to the potential presence of wild ramification, the same computation does not seem to work in the general settings that we consider.  If the index of the canonical cover is prime-to-$p$, then we essentially discussed this generalization in \autoref{cor.PerfdPureImpliesLC}.  We now take a more general approach as we hope that \autoref{conj.GeneralQuasiTorsorPerfdPure} is true.

Before we begin, we remind the reader that a scheme is deminormal if it is S2, and is either nonsingular or has ``nodes'' after localizing at points of codimension 1, see \cite[Definition 5.1]{KollarKovacsSingularitiesBook}.  Note this forces the scheme to be Gorenstein in codimension 1 as well, and so allows us to use the techniques of divisors when working with the canonical module.

\begin{proposition}
    \label{prop.GeneralCoverResultForLC}
    Suppose $(R, \fram)$ is a Noetherian reduced local ring with a dualizing complex and $R \subseteq S$ is a finite extension such that $R,S$ are locally equidimensional and with $f : \Spec S \to \Spec R$ the induced map.  Suppose $R$ is deminormal  and $\bQ$-Gorenstein, and that $S$ is S2, and quasi-Gorenstein.  Additionally fix $\omega_R \subseteq K(R)$ and suppose that $(\omega_R^{-1} \otimes_R S)^{**} = y S$ for some $y \in K(S)$ (or equivalently, $f^* (-K_R) \sim 0$).   Further suppose that the twisted Grothendieck trace map:
    \[
        \Phi : S \cong S(K_S - f^*K_R) = y \cdot \omega_S \to R(K_R - K_R) = R
    \]
    is surjective.  
    Suppose that for each birational $\mu : Y \to \Spec S$ satisfying the following conditions:
    \begin{enumerate}
        \item  {$\mu$} is an isomorphism outside a set $V(J) \subseteq \Spec S$ of codimension $\geq 2$,
        \item $Y$ is G1 and S2,
        \item If $F = \mu^{-1}(V(J))_{\red}$, we have that $F$ has pure codimension 1 and that $Y$ is regular at each generic point of $F$ (that is, $F$ can be viewed as a divisor),
    \end{enumerate}
    we have that 
    \[ 
        \mu_* \cO_Y(K_Y + F) = \mu_* \sHom_Y(\sI_F, \omega_{Y}) \to \Hom_S(J, \omega_S) = \omega_S
    \] 
    is surjective and hence an isomorphism.  Then $R$ is semi-log canonical.
\end{proposition}
\begin{proof}    
    We first explain the twisted Grothendieck trace map mentioned in the statement.  The Grothendieck trace is the evaluation-at-1 map $\omega_S = \Hom_R(S, \omega_R) \to \omega_R$.  Tensoring with $\omega_R^{-1}$ and reflexifying/S2-ifying gives us a map
    \[
        y \cdot \omega_S = (\omega_S \otimes_R \omega_R^{-1})^{**} \to (\omega_R \otimes \omega_R^{-1})^{**} = R.
    \]
    As $S$ is quasi-Gorenstein, $y \cdot \omega_S \cong S$ and we have described our map $\Phi$.

    To show that $R$ is semi-log canonical, it suffices to show that for each prime divisor $D$ appearing on some normal birational model, and which is exceptional over the normalization of $R$, that $D$ has discrepancy $\geq -1$.  Applying \autoref{thm.ArtinZariskiAbhyankar}, we can obtain a blowup $\pi : X \to \Spec R$ with a prime divisor on $X$ whose generic point is a valuation ring which corresponds $D$ (this divisor we thus also call $D$) and which is an isomorphism over $U = \Spec R \setminus \pi(D)$.  It thus suffices to compute the discrepancy of $D$ on $X$.

    Suppose that $I \subseteq R$ is an ideal 
    whose blowup produces $\pi : X \to \Spec R$ (in particular, $I$ is invertible when restricted to $U$).  Let $Y_0 \to \Spec S$ denote the blowup of $IS$ and note we have a finite map $Y_0 \to X$.  Let $V \subseteq Y_0$ denote the inverse image of $U$.  Observe that $V$ is quasi-Gorenstein (as it is also an open subset of $\Spec S$), and let {$i : V \to Y_0$} denote the inclusion.  Consider $\sC$ the integral closure of $\cO_{Y_0}$ in $i_* \cO_V$, in other words $\sC = \cO_{Y_0}^{\mathrm{N}} \cap i_* \cO_V$ where the intersection takes place in the fraction field of {$Y_0$}.  Set 
    \[
        Y := \sheafspec_{Y_0} \big( \sC \big).
    \]
    We see that $Y$ is G1 and S2 and has a finite map to $Y_0$, as our base is excellent, and hence has a finite map $g : Y \to X$.  Furthermore the induced map $Y \to \Spec S$ is an isomorphism over $V$.  Let $E$ and $F$ denote the reduced exceptional sets of the maps $X \to \Spec R$ and $Y \to \Spec S$ respectively.  By \autoref{thm.ArtinZariskiAbhyankar} we see that $X$ is regular at all generic points of $E$, and by construction, $Y$ is regular at all generic points of $F$.  Thus $E$ and $F$ are divisors in the sense of \cite{KollarKovacsSingularitiesBook}, \cf \cite{HartshorneGeneralizedDivisorsOnGorensteinSchemes}.

    We have the following commutative diagram:
    \[
        \xymatrix{
            F \ar@{^{(}->}[d] \ar@{->>}[r]^h & E \ar@{^{(}->}[d] \\
            Y \ar[d]_{\mu} \ar@{->>}[r]^{g} & X \ar[d]^{\pi} \\
            \Spec S \ar@{->>}[r]_{f} & \Spec R
        }
    \]
    Note the horizontal maps are all finite by construction.  We obtain the following induced map of canonical modules:
    \[
        \xymatrix{
            0 \ar[r] & g_* \omega_Y \ar[d]  \ar[r] & g_* \omega_Y(F) \ar[d] \ar[r] & h_* \omega_F \ar[d] \\
            0 \ar[r] & \cO_X(K_X) \ar[r] & \cO_X(K_X + E) \ar[r] & \omega_E\\
        }
    \]
    where $\sHom_Y(\sI_F, \omega_Y) = \omega_Y(F) = \cO_Y(K_Y + F)$, notation is reasonable as $Y$ is G1 and S2.
    \begin{claim}
        The image of $g_*\big(y \cdot \sHom(\sI_F, \omega_Y)\big) \to \cO_X(K_X + E)$ is contained in the sheaf $\cO_X(\lceil K_X + E - \pi^* K_R\rceil)$.
    \end{claim}
    \begin{proof}[Proof of claim]
        Since all sheaves are S2, it suffices to check this in codimension 1.  The claim holds on $V$ as we already asserted a version of it for $\Spec S \to \Spec R$ when describing the twisted Grothendieck trace map.  Over the generic points of $E$ (that is, at the generic points of $F$), $Y$ is normal and the claim is straightforward with our choice of rounding. 
    \end{proof}
    
    Pushing forward to $\Spec R$, we obtain
    \[
        \pi_* g_* \big(y \cdot \sHom(\sI_F, \omega_Y)\big) \to \pi_* \cO_X(\lceil K_X + E - \pi^* K_R\rceil) \to R.
    \]
    We can also factor this map alternately as:
    \[
        \pi_* g_* \big(y \cdot \sHom(\sI_F, \omega_Y)\big) = f_* \mu_*  \big(y \cdot \sHom(\sI_F, \omega_Y)\big) \to f_* (y \cdot \omega_S) \to R
    \]
    which is surjective as it is a composition of surjective maps (by hypothesis).  It follows that 
    \[
        \pi_* \cO_X(\lceil K_X + E - \pi^* K_R\rceil) \to R
    \]
    surjects.  
    
    We claim this implies that $R$ is log canonical.  Indeed, if $R$ is not log canonical it has exceptional divisors with arbitrarily negative discrepancies (on some blowup).  In particular, if we have a discrepancy $\leq -2$, then $\pi_* \cO_X(\lceil K_X + E - \pi^* K_R \rceil) \subsetneq R$ on any birational model exhibiting that discrepancy.
\end{proof}

We now state our more general version of \autoref{prop.KSSPushForwardNormalCase} outside of the normal case. 

\begin{theorem}
    \label{thm.PerfdPurePlusQGorIndexImpliesLC}
  Suppose $(R, \fram)$ is a Noetherian local ring with a dualizing complex of mixed characteristic $(0, p > 0)$.  If $R$ is S2, deminormal, $\bQ$-Gorenstein and has a canonical cover $S$ that is \absperfdinj (for instance, if $R$ is \claperfdpure and $\bQ$-Gorenstein of index not divisible by $p$, or assuming \autoref{conj.GeneralQuasiTorsorPerfdPure}), then $R$ is semi-log canonical.
\end{theorem}

\begin{proof}
        We assume we have a cyclic cover $S =  \bigoplus_{i = 0}^{n-1} R(-iK_R)$ where $n$ is the {Cartier} index of $K_R$, and where the multiplication on $S$ is defined using an isomorphism $\omega_R^{(n)} \cong R$.  Note, that we do not know that the cyclic cover $S$ is normal, but it is certainly S2 and G1 (Gorenstein in codimension 1).  By hypothesis, some such $S$ is \absperfdinj (in the case that the index is not divisible by $p > 0$, this is \autoref{prop.CyclicCoverNIndexStuff}).  Therefore, by \autoref{prop.KSSPushForwardNormalCase}, we see that $\mu_* \cO_Y(K_Y + F) \to \omega_S$ surjects for any $\mu$ satisfying the conditions of \autoref{prop.GeneralCoverResultForLC}.   Furthermore, $R \to S$ is split so that the twisted Grothendieck trace map $S(K_S - f^* K_R) \to R(K_R - K_R) = R$ surjects.  
        
        Hence, we may apply \autoref{prop.GeneralCoverResultForLC} and so deduce that $R$ is semi-log canonical.
\end{proof}

\section{Inversion of adjunction}

    In this section, we are primarily interested in the following question.  If $(R, \fram)$ is local, $0 \neq f \in \fram$ is a nonzerodivisor, and $R/(f)$ is \absperfdinj or perfectoid-injective, is $R$ likewise?  In characteristic $p > 0$, this is open in full generality with the Cohen-Macaulay case being shown in \cite{FedderFPureRat}, and with other substantial progress on this question found for instance in \cite{HoriuchiMillerShimomoto,MaQuyFrobeniusActionsAndDeformation}.  In characteristic zero, the analogous result for Du Bois singularities is shown in \cite{KovacsSchwedeDBDeforms}.

    We will prove a slightly stronger statement (also analogous to the results in characteristic zero and $p> 0$) when $R$ is LCI, and for that we need the following definition.

    \begin{definition}
    \label{def.PerfdPurePair}
        Let $R$ be a Noetherian ring with $p$ in its Jacobson radical, and $f\in R$ a nonzerodivisor.   We say that the pair $(R,f)$ is \emph{\claperfdpure} if there is a choice of perfectoid $R$-algebra $B$ containing a (fixed choice of) compatible system of $p$-power roots of $f$ in $B$, such that the map
        \[ 
            fR\to (fB)_{\perfd} = (f^{1/p^{\infty}}B)^-
        \]      
        is pure as a map of $R$-modules (see \cite[Lemma 2.3.2]{CaiLeeMaSchwedeTucker} or \cite[Section 7]{BhattScholzepPrismaticCohomology}) for the equality above). {Here, $I^-$ denotes the $p$-adic closure of an ideal $I$.}  In the same setting, we say that $(R, f)$ is \emph{\claperfdinj} if 
        \[
            H^i_{\fram}(fR) \to H^i_{\fram}((fB)_{\perfd})
        \]
        injects for every $i$ and every maximal ideal $\fram$.

       {Finally, we define $(f)_{\perfd}$ to be the fiber of the map $R_{\perfd} \to (R/fR)_{\perfd}$ in $D(R)$.}
        We say that $(R, f)$ is \emph{\absperfdpure} if the induced map $fR\to (f)_\perfd$ is pure in $\myD(R)$, and that $(R, f)$ is
        \emph{\absperfdinj} if the induced map
        \[
            H^i_{\fram}(fR) \to H^i_{\fram}((f)_{\perfd})
        \] 
        is injective for every $i$ and every maximal ideal $\fram$.
    \end{definition}
 
    \begin{remark}
        In this paper, we restrict ourselves to pairs with integer coefficients.  A non-empty subset of the authors plans to explore pairs with rational coefficients in a future work.
    \end{remark}

    \begin{remark} \label{remark:pinj-log-to-usual}
        Note if $(R, f)$ is perfectoid injective (respectively perfectoid pure), then $R$ is also.  This follows since we have a factorization:
        \[
            R \to B \xrightarrow{1 \mapsto f} (fB)_{\perfd}
        \]
        which can be identified with $fR \to (fB)_{\perfd}$.
    \end{remark}    

    For us, we will only be working in the case that $R$ is Cohen-Macaulay, and so $H^i_{\fram}(fR) \cong H^i_{\fram}(R) = 0$ for $i < d = \dim R$.

\begin{lemma}
If $(R,f)$ is perfectoid injective (resp.\  perfectoid pure) then $(R,f)$ is lim-perfectoid injective (resp.\ lim-perfectoid pure).
\end{lemma}
\begin{proof}
For any perfectoid $R$-algebra $B$ that contains a compatible system of $p$-power roots of $f$, we have an exact sequence \[0\to (f^{1/p^\infty}B)^-\to B\to B/(f^{1/p^\infty}B)^-\to 0\]
Since $B/(f^{1/p^\infty}B)^- \cong (B/(f))_{\perfd}$ is perfectoid, we have maps of fiber sequences
\[
	\xymatrix{
		 fR\ar[r]\ar[d] & R \ar[r]\ar[d] & R/fR \ar[d] \\
        (f)_{\perfd}\ar[r]\ar[d] & R_{\perfd}\ar[r]\ar[d]& (R/f)_{\perfd}\ar[d]\\
		 (f^{1/p^{\infty}}B)^-\ar[r] &B\ar[r] & B/(f^{1/p^{\infty}}B)^-
}
\]
The factorization of the left hand column yields the required injectivity (resp.\  purity) by \autoref{lem:first_factor_is_pure}.
\end{proof}

\begin{proposition}
Let $(R,\m)$ be a Noetherian Cohen-Macaulay local ring of residue characteristic $p>0$ and $f\in R$ a nonzerodivisor.  If the pair $(R,f)$ is \claperfdinj (resp.\ \absperfdinj) then $R/fR$ is \claperfdinj (resp.\ \absperfdinj).
\end{proposition}
\begin{proof}
Let $B$ be a perfectoid $R$-algebra such that $f$ has a compatible system of $p$-power roots $\{f^{1/p^e}\}_e$ and such that the natural map $fR\to (f^{1/p^{\infty}}B)^-$ is pure.  Then we have a commutative diagram with exact rows:
	\[
	\xymatrix{
		0\ar[r] & fR\ar[r]\ar[d] & R \ar[r]\ar[d] & R/fR \ar[r]\ar[d] & 0\\
		0\ar[r]& (f^{1/p^{\infty}}B)^- \ar[r] &B\ar[r] & B/(f^{1/p^{\infty}}B)^-\ar[r] & 0
}
\]
to which taking top local cohomology gives a diagram with exact rows
	\[
\xymatrix{
	0\ar[r] & H^{d-1}_{\m}(R/fR)\ar[r]\ar[d] & H^d_{\m}(fR)\ar[r]\ar[d] & H^d_{\m}(R) \ar[r]\ar[d] & 0\\
 & H_\m^{d-1}(B/(f^{1/p^{\infty}}B)^-) \ar[r] &H^d_{\m}((f^{1/p^\infty}B)^-)\ar[r] & H^d_{\m}(B) \ar[r] &  0
}
\]
The middle vertical arrow is injective since $(R,f)$ is \claperfdinj, and thus the left vertical arrow is injective by an obvious diagram chasing. Since $B/(f^{1/p^\infty}B)^-$ is perfectoid, this means $R/fR$ is \claperfdinj.

In the  lim-perfectoid injective case, we have a fiber sequence $(f)_{\perfd}\to R_{\perfd}\to (R/f)_{\perfd}$, and therefore have a map of fiber sequences 
\[
	\xymatrix{
		fR\ar[r]\ar[d] & R \ar[r]\ar[d] & R/fR \ar[d] \\
		 (f)_{\perfd}\ar[r] &R_{\perfd}\ar[r] & (R/f)_{\perfd}
}
\]
Taking local cohomology we have
\[
\xymatrix{
	0\ar[r] & H^{d-1}_{\m}(R/fR)\ar[r]\ar[d] & H^d_{\m}(fR)\ar[r]\ar[d] & H^d_{\m}(R) \ar[r]\ar[d] & 0\\
 & H_\m^{d-1}((R/f)_\perfd) \ar[r] &H^d_{\m}((f)_{\perfd})\ar[r] & H^d_{\m}(R_{\perfd}) \ar[r] &  0
}
\]
The middle vertical arrow is injective since $(R,f)$ is \absperfdinj, and thus the left vertical arrow is injective by an obvious diagram chasing.  This shows that $R/fR$ is \absperfdinj. 
\end{proof}

We next prove the converse of the proposition above when $R$ is LCI. Note that in this case, by \autoref{lem:pure_injective_comparison} and \autoref{thm: lci case}, all four notions (\claperfdpure, \absperfdpure, \claperfdinj, and \absperfdinj) are equivalent and so we may replace \claperfdinj in the theorem below by any of the other three notions.

\begin{theorem}
\label{thm: inv of adj}
Let $(R,\m, k)$ be a Noetherian local ring of residue characteristic $p>0$. Suppose $R$ is a complete intersection. Let $f\in R$ be a nonzerodivisor such that $R/fR$ is \claperfdinj (e.g., $R/fR$ has characteristic $p>0$ and is $F$-injective). Then $(R,f)$ is \claperfdinj, and thus $R$ is \claperfdinj. 
\end{theorem}
\begin{proof}
We may assume $R$ is complete by \autoref{lem:pfdpurecompletion}. By Cohen's structure theorem we can write $R=S/(f_1,\dots,f_c)$ such that $S$ is a complete unramified regular local ring of mixed characteristic $(0,p)$ with $f$ being part of a regular system of parameters { of $R$} and $f_1,\dots,f_c$ being a regular sequence on $S$. We fix an isomorphism $S\cong C_k[[f, x_2,\dots, x_n]]$ for $C_k$ a Cohen ring\footnote{a complete unramified mixed characteristic DVR with residue field $k$} (note that even if $f=p$ in $R$, we can still take $S$ in this form and let one of the $f_i$'s be $f-p$). Note that, with this isomorphism, $S/fS$ is still a complete unramified regular local ring and we have $R/fR$ is the quotient of $S/fS$ by the image of $f_1,\dots,f_c$ (which is a regular sequence in $S/fS$).  
Let $S_\infty$ be the $p$-adic completion of  $W(k^{1/p^\infty})[[f, x_2,\dots, x_n]][p^{1/p^\infty},f^{1/p^\infty},x_2^{1/p^\infty},\dots, x_n^{1/p^\infty}]$ and $(S/fS)_\infty$ be the $p$-adic completion of $W(k^{1/p^\infty})[[x_2,\dots, x_n]][p^{1/p^\infty},x_2^{1/p^\infty}, \dots, x_n^{1/p^\infty}]$. It is straightforward to check using the universal property of perfectoidization that $$(R/fR)^{(S/fS)_\infty}_{\perfd} \cong R^{S_\infty}_\perfd/(f^{1/p^\infty}R^{S_\infty}_\perfd)^-$$
where again $(R/fR)^{(S/fS)_\infty}_{\perfd} := ((R/fR) \otimes_{S/fS} {(S/fS)_\infty})_{\perfd}$.
By \autoref{thm: lci case}, we have the following commutative diagram
\[
\xymatrix{
0\ar[r] & H_\m^{d-1}(R/fR) \ar[r] \ar[d] & H_\m^d(fR) \ar[r] \ar[d] & H_\m^d(R) \ar[r] \ar[d] & 0 \\
0\ar[r] & H_\m^{d-1}((R/fR)^{(S/fS)_\infty}_{\perfd}) \ar[r] & H_\m^d((f^{1/p^\infty}R^{S_\infty}_\perfd)^-) \ar[r] & H_\m^d(R^{S_\infty}_\perfd) \ar[r] & 0
}.
\]
By our assumption and \autoref{lem:RAperfdsuffices}, the left vertical map in the above diagram is injective. So chasing this diagram with the socle representative of $H_\m^d(fR)$ shows that the middle map 
is injective. Thus the pair $(R,f)$ is \claperfdinj as wanted {(see Remark \ref{remark:pinj-log-to-usual})}.  
\end{proof}

When $f=p$, we have the following proposition which is an analog of a weak version of results in \cite{FedderWatanabe} and \cite{MaSchwedeShimomoto}.
We refer the reader to \cite{MaSchwedeSingularitiesMixedCharBCM} or \cite{CaiLeeMaSchwedeTucker} for the definition and basic properties of BCM-regularity. 

\begin{proposition}
Let $(R,\m)$ be a Noetherian complete local domain of mixed characteristic $(0,p>0)$. Suppose $R$ is a complete intersection, $R/p$ is $F$-pure, and $R[1/p]$ is regular. Then $(R,(1-\epsilon)\Div(p))$ is BCM-regular for all ${ 0 < } \epsilon\ll1$. 
\end{proposition}
\begin{proof}
Let $J$ be the ideal of $R$ generated by all elements $g$ such that $A[1/g]\to R[1/g]$ is finite \'etale for some $A\to R$ Noether-Cohen normalization. If $p\notin \sqrt{J}$, then we can find a prime $Q\supseteq J$ such that $p\notin Q$. Since $R[1/p]$ is regular, it follows that $R_Q$ is regular. But then by \cite[Theorem 0.1]{HeitmannEtaleLocus}, there exists a Noether-Cohen normalization $A\to R$ and $g\notin Q$ such that $A[1/g]\to R[1/g]$ is \'etale contradicting our choice of $g$. It follows that there are $A_i\to R$, $1\leq i\leq n$, Noether-Cohen normalizations such that $A_i[1/g_i]\to R[1/g_i]$ is finite \'etale and $p\in \sqrt{(g_1,\dots,g_n)}$. 

We choose a complete and unramified regular local ring $S$ such that $A_i\to R$ factors through $A_i\to S\to R$ for all $i$ (simply add an indeterminate for each indeterminate in each of the $A_i$) and we may further choose maps $R^{(A_i)_\infty}_\perfd \to R^+$ that factor through $R^{S_\infty}_\perfd$. Suppose $(R,(1-\epsilon)\Div(p))$ is not BCM-regular, then by definition (since $R$ is Gorenstein), for the socle representative $\eta\in H_\m^d(R)$, we have $p^{1-\epsilon}\eta =0$ in $H_\m^d(B)$ for all sufficiently large perfectoid big Cohen-Macaulay $R^+$-algebra $B$. By \cite[Lemma 5.1.6]{CaiLeeMaSchwedeTucker}, for each $i$ we know that $(g_i)_\perfd p^{1-\epsilon}\eta =0$ in $H_\m^d(R^{(A_i)_\infty}_\perfd)$. But then we know that $(g_i)_\perfd p^{1-\epsilon}\eta =0$ in $H_\m^d(R^{S_\infty}_\perfd)$ for all $i$ since $R^{(A_i)_\infty}_\perfd$ maps to $R^{S_\infty}_\perfd$. It follows that $(p^{1/p^\infty})p^{1-\epsilon}\eta =0$ in $H_\m^d(R^{S_\infty}_\perfd)$ since $p\in \sqrt{(g_1,\dots,g_n)}$. Thus, the map $R\to R^{S_\infty}_\perfd$ sending $1$ to $p^{1-\epsilon'}$ is not pure for all $\epsilon'<\epsilon\ll1$. But by \autoref{thm: inv of adj}, $(R, \Div(p))$ is \claperfdpure and thus $R\to (p^{1/p^\infty})R^{S_\infty}_\perfd$ sending $1$ to $p$ is pure, which is a contradiction. 
\end{proof}

\section{Examples}

In this section we provide some examples of \claperfdpure singularities.  

\begin{example}
    Suppose $R = \bZ_p\llbracket x_1, \dots, x_n\rrbracket/(f_1, \ldots, f_c)$ where $f_1, \dots, f_c$ form a regular sequence.  If $R/(p)$ is $F$-pure, then $R$ is \claperfdpure by \autoref{thm: inv of adj}.  In particular, $\bZ_p\llbracket x, y, z\rrbracket/(x^3+y^3+z^3)$ is \claperfdpure for $p \equiv 1 \mod{3}$.
\end{example}

By using Rees algebras as in \cite{MaSchwedeTuckerWaldronWitaszekAdjoint}, we can generalize the previous example to the case where one of the variables is replaced by $p$.  

 
\begin{proposition}
\label{prop:examplereplacingvariablebyp}
    Fix a prime $p>0$ and $k$ a perfect field of characteristic $p$. Suppose that $f_1, \ldots, f_c \in \bZ[x_1, \ldots, x_n]$ are homogeneous polynomials of positive degree (with respect to the standard grading with $\deg(x_i)=1$ for all $i$). Let 
    \[
    R = W(k)\llbracket x_1, \ldots, x_n \rrbracket / (x_1 - p, f_1, 
    \ldots, f_c)
    \]
    and let $\fram = ( p,x_1,\ldots, x_n ) \subseteq R$ be the maximal ideal of $R$. If $\mathrm{gr}_\fram(R)$
    is an $F$-pure complete intersection, then $R$
    is \claperfdpure. 
\end{proposition}

\begin{remark}
An important situation where \autoref{prop:examplereplacingvariablebyp} applies is as follows. With notation as above, suppose that $R$ has dimension $d = n - c$ and 
    \(
    k [ x_1, \ldots, x_n ] / (f_1, \ldots, f_c)
    \) is an $F$-pure complete intersection of dimension $d$ that is additionally a domain. In this case, consider the natural surjection of graded rings
    \[
    \phi \colon k[ x_0, x_1, \ldots, x_n ] \twoheadrightarrow \mathrm{gr}_\fram(R)
    \]
    with $x_0, x_1, \ldots, x_n$ mapping to the classes of $p,x_1, \ldots, x_n$ in $\fram/ \fram^2$, respectively. As $f_1, \ldots, f_c \in \bZ[x_1, \ldots, x_n]$ are homogeneous, observe that $(x_0-x_1,f_1, \ldots, f_c) \subseteq \ker \phi$, giving a surjection
    \[
   k[x_1, \ldots, x_c] / (f_1, \ldots, f_c) \cong k [x_0, x_1, \ldots, x_n] / (x_0-x_1,f_1, \ldots, f_c) \twoheadrightarrow \mathrm{gr}_\fram(R).
    \]
    Since \(
    k [ x_1, \ldots, x_n ] / (f_1, \ldots, f_c)
    \) is a domain with dimension $d = \dim R = \dim \mathrm{gr}_\fram(R)$, it follows that this map is an isomorphism, and hence $\mathrm{gr}_\fram(R) \cong k [ x_1, \ldots, x_n ] / (f_1, \ldots, f_c)$ so that the assumptions of \autoref{prop:examplereplacingvariablebyp} are satisfied. See \cite{DeStefaniRossiVarbaro} for additional recent work on computing $\mathrm{gr}_\fram(R)$ in mixed characteristic.
\end{remark}
     
\begin{proof}
    Let $T = R[\fram t,t^{-1}]$ be the extended Rees algebra with $\mathfrak{n} = (t^{-1},\fram T, \fram t)$ its homogeneous maximal ideal. Note that $t^{-1}$ is a non-zero divisor, and that  $T / (t^{-1}) \cong \mathrm{gr}_\fram(R)$. In particular, $T$ is a complete intersection ring, and $T_\mathfrak{n}/(t^{-1})$ is $F$-pure. Thus, by \autoref{thm: inv of adj}, it follows that $T_\mathfrak{n}$ is \claperfdpure. By \autoref{lem:Puremap}, the conclusion follows provided that $R \to T_\mathfrak{n}$ is pure.

    Since $T_\mathfrak{n}$ is \claperfdpure, it is necessarily reduced. As the associated primes of $T$ are all homogeneous \cite[Lemma 1.5.6 (b) (ii)]{BrunsHerzog}, it follows that $T$ and hence also $R$ are reduced as well. By \cite{HochsterCyclicPurity}, $R \to T_\mathfrak{n}$ is pure if and only if $R/I \to T_\mathfrak{n} / I T_\mathfrak{n}$ is injective for all $\fram$-primary ideals $I \subseteq R$. Given such an $I$, pick $\ell \gg 0$ so that $\fram^\ell \subseteq I$ and set $J = t^{-\ell}T + IT + \fram^\ell t^\ell T$. Then $J$ is a homogeneous $\mathfrak{n}$-primary ideal of $T$ with $[J]_0 = I$, so that the natural map $R/I \to T / JT$ is a split injection. In lieu of the factorization
    \[
    R/I \to T_\mathfrak{n} / I T_\mathfrak{n} \to T_\mathfrak{n} / J T_\mathfrak{n} \cong T /J T,
    \]
    we see that $R/I \to T_\mathfrak{n} / I T_\mathfrak{n}$ is a split injection as well.
\end{proof}

\begin{example}[Calabi-Yau-like hypersufaces]
\label{ex:CYlikehypersurfaces}
Fix $p > 0$ a prime and $k$ a perfect field of characteristic $p > 0$.  Suppose $f \in \bZ[x_1, \dots, x_n]$ is a homogeneous equation of degree $\leq n$ none of whose coefficients are divisible by $p$.  This gives us a hypersurface singularity:
    \[
        R = W(k)\llbracket x_2, \dots, x_n\rrbracket / (f(p, x_2, \dots, x_n) ).
    \]
    Heuristically, we are taking a cone over a smooth hypersurface, but we replaced one of the variables with $p$.   Suppose the corresponding hypersurface singularity 
    \[
        k[x_1, \dots, x_n]/(f(x_1, \dots, x_n))
    \]
    is an $F$-pure domain in characteristic $p > 0$.  Then we see by the proposition that $R$ is \claperfdpure.

    For example, 
    $\bZ_p[y,z]/(p^3+y^3+z^3)$ is \claperfdpure for $p \equiv 1 \mod 3$.

\end{example}

\subsection{Frobenius liftable singularities}
Since our inversion of adjunction applies only for complete intersections, it is not so easy to construct examples of \claperfdinj singularities which are neither complete intersections nor splinters. In what follows, we show that quasi-Gorenstein Frobenius liftable singularities are \claperfdinj. In particular, cones over canonical lifts of ordinary abelian varieties are \claperfdinj.  Note we have learned a similar related construction will appear in forthcoming work of Ishizuka and Shimomoto \cite{ShimomotoIshizuka.Quasi-canonicalLifting}.

\begin{proposition}
Let $k$ be a perfect field of characteristic $p>0$ and let $R$ be a Noetherian local domain containing $W(k)$ such that $p$ is contained in the maximal ideal of $R$. Let $R_{p=0}$ be the reduction of $R$ modulo $p$. Assume that 
\begin{enumerate}
    \item $\omega_{R} \simeq R$ and $\omega_{R_{p=0}} \simeq R_{p=0}$,
    \item $R_{p=0}$ is $F$-split, and
        \item there exists a finite ring homomorphism $\mathscr{F} \colon R \to R$ such that modulo $p$ the homomorphism $\mathscr{F}$ agrees with Frobenius $F \colon R_{p=0} \to R_{p=0}$ and such that the following diagram commutes\footnote{In fact, the commutativity is automatic by deformation theory if $R$ is $p$-complete - a case we can reduce to.  Concretely $W(k)$ is generated by Teichm\"uller lifts of elements of $k$, and these are elements that admit all $p$-power roots; any such element in any $p$-complete $\delta$-ring must be killed by $\delta$, see \cite[Lemma 2.32]{BhattScholzepPrismaticCohomology}.}:
    \[
    \xymatrix{
      R \ar[r]^{\mathscr{F}} & R \\
      W(k) \ar@{^{(}->}[u] \ar[r]^F & W(k). \ar@{^{(}->}[u]
    }
    \]
\end{enumerate}
Then $R$ is \claperfdpure.
\end{proposition}

\begin{proof}
We construct a natural perfectoid cover of $R$ associated to $\mathscr F$. Let 
\[
S := \varinjlim\ (R \xrightarrow{\mathscr F} \mathscr{F}_*R \xrightarrow{\mathscr F} \mathscr{F}^2_*R \to \cdots ),
\]
let
\[
R^{\rm nc}_{\infty} := S[p^{1/p^\infty}] = S \otimes_{W(k)} W(k)[p^{1/p^{\infty}}]
\]
and let 
\[
R_{\infty} := {R^{\rm nc}_{\infty}}^{\wedge_p}.
\]

\begin{claim} \label{claim:Rinfty-perfectoid} $R_{\infty}$ is a perfectoid ring.
\end{claim}
\begin{proof}[Proof of Claim]
Set $\varpi = p^{1/p}$. Since $R_\infty$ is $p$-torsion free and $p$-adically complete, it is enough to show that the Frobenius $F \colon R_{\infty}/\varpi \to R_{\infty}/\varpi$ is surjective (see \cite[Lemma 3.10]{BhattMorrowScholzeIHES}). This is immediate by construction as $F \colon S/p \to S/p$ is surjective.
\end{proof}

\begin{claim} \label{claim:Frob-splits} $\mathscr{F} \colon R \to \mathscr{F}_* R$ splits.
\end{claim}
\begin{proof}
Consider the following diagram
\[
\xymatrix{
  0 \ar[r] & \mathscr{F}_*R \ar[r]^{\cdot p} & \mathscr{F}_*R \ar[r] & F_*R_{p=0} \ar[r] & 0 \\
  0 \ar[r] & R \ar[r]^{\cdot p} \ar[u]^{\mathscr{F}}  & R \ar[r] \ar[u]^{\mathscr{F}} & R_{p=0} \ar[u]_{F} \ar[r] & 0
}
\]
where the left horizontal maps are given by multiplication by $p$. Now apply $\Hom_R(-,\omega_R)$ to get the following diagram:
\begin{equation} \label{eq:diagram-Fr-lift}
\xymatrix{
 0 \ar[r] & \mathscr{F}_*\omega_R \ar[d]^{\Tr_{\mathscr{F}}} \ar[r]^{\cdot p} & \mathscr{F}_*\omega_R \ar[d]^{\Tr_{\mathscr{F}}} \ar[r] & F_*\omega_{R_{p=0}} \ar[d]^{\Tr_{F}} \\
 0 \ar[r] & \omega_R \ar[r]^{\cdot p}  & \omega_R \ar[r]  & \omega_{R_{p=0}}.
}
\end{equation}

Here, the structure of the top row is a consequence of the following identities (we would always assume that $\omega_R^\bullet$ is normalized so that it sits in $D^{[-\dim(R), 0]}$): 
\begin{enumerate}
\item $\myR\Hom_R(\mathscr{F}_*R, \omega^\mydot_R) \simeq \mathscr{F}_*\myR\Hom_R(R, \mathscr{F}^!\omega^\mydot_R) \simeq \mathscr{F}_*\omega^{\mydot}_R$,  which implies that
\[
\Hom_R(\mathscr{F}_*R, \omega_R) \simeq \mathscr{F}_*\omega_R.
\]
\item $\myR\Hom_R(F_*R_{p=0}, \omega^\mydot_R) \simeq {\myR\Hom_{R_{p=0}}(F_*R_{p=0}, \omega^\mydot_{R_{p=0}}) \simeq} F_*\omega^{\mydot}_{R_{p=0}}$, which implies that
\[
\Ext^1_R(F_*R_{p=0}, \omega_R) \simeq F_*\omega_{R_{p=0}}.
\]
\end{enumerate}
Since $\omega_R\cong R$ and $\omega_{R_{p=0}} \cong R_{p=0}$, the horizontal map $\omega_R \to \omega_{R_{p=0}}$ can be identified with the restriction map $R \to R_{p=0} = R/p$. In particular, the rightmost horizontal arrows in Diagram (\ref{eq:diagram-Fr-lift}) are surjective. 
The right most square of our diagram can thus be reinterpreted as follows
\[
    \xymatrix{
        \sF_* R \ar[d]_{\Tr_{\sF}}\ar@{->>}[r] & F_* R_{p=0}\ar[d]_{\Tr_F}\\
        R \ar@{->>}[r] & R_{p=0}.
    }
\]
As $\Tr_F$ is surjective, its image does not land in the maximal ideal of $R_{p=0}$.  Hence the image of $\Tr_{\sF}$ also does not land in the maximal ideal of $R$ and so $\Tr_{\sF}$ is surjective.  This implies $\sF$ splits and proves the claim.
\end{proof}
By Claim \ref{claim:Frob-splits}, we immediately get that the inclusion $R \hookrightarrow S$ is pure. Moreover, the inclusion $S \hookrightarrow R^{\rm nc}_{\infty}$ is pure, because $R^{\rm nc}_{\infty}$ is a colimit of free modules over $S$. 
Finally, the composition $R \hookrightarrow R_{\infty}^{\mathrm nc} \to R_\infty$ is pure since we can check purity by tensoring with $E$. Since $R_{\infty}$ is perfectoid by Claim \ref{claim:Rinfty-perfectoid}, the proof that $R$ is \claperfdpure is concluded.
\end{proof}

We say that a $d$-dimensional scheme $X$ defined over a positive characteristic field is \emph{weakly ordinary} if the action of Frobenius $F^* \colon H^d(X,\cO_X) \to H^d(X,\cO_X)$ on the highest cohomology of the structure sheaf is bijective. When $X$ is Cohen-Macaulay and $\omega_X$ is trivial, this is equivalent to $X$ being globally $F$-split.
\begin{example}
Let $X$ be a smooth projective variety defined over a perfect field $k$ of characteristic $p>0$ such that $\Omega^1_X$ is trivial. Assume that $X$ is weakly ordinary and let $\mathscr X$ be the canonical lift of $X$ over $W(k)$ as in \cite[Appendix: Theorem (1)]{Mehta1987}. Suppose that $\omega_{\mathscr X}$ is trivial (this is for example the case when $X$ is an ordinary abelian variety). Let $A$ be an ample line bundle on $X$ and let $\mathscr A$ be the canonical lift of $A$ as in \cite[Appendix: Theorem (3)]{Mehta1987}. Finally, let $R$ be the cone of $\mathscr X$ with respect to some very ample multiple of $\mathscr A$. 

Then by the above proposition $R$ is \claperfdinj. Indeed, by \cite[Appendix: Theorem (1)]{Mehta1987}, there exists a morphism $\mathscr F \colon \mathscr X \to \mathscr X$ over the Frobenius $F \colon \Spec W(k) \to \Spec W(k)$ such that $\mathscr F$ agrees modulo $p$ with the Frobenius morphism on $X$. Moreover, by \cite[Appendix: Theorem (3)]{Mehta1987}:
\[
\mathscr{F}^* \mathscr A = \mathscr{A}^p.
\]
In particular, there exists an induced ring homomorphism $\mathscr{F} \colon R \to R$ which agrees with the Frobenius $F \colon R_{p=0} \to R_{p=0}$ on the reduction $R_{p=0}$ of $R$ modulo $p$. Moreover, $\omega_{R_{p=0}}$ and $\omega_R$ are trivial by construction.
\end{example}

At this point we are unable to construct an example of a non-splinter \claperfdinj ring which is neither a complete intersection nor arises from Frobenius liftable examples in equal characteristic $p > 0$. However, it seems natural to expect that cones over Serre-Tate type lifts of Calabi-Yau varieties are \claperfdinj. For example, one can ask the following.
\begin{question}
Let $X$ be an ordinary K3 surface over a perfect field $k$ of characteristic $p>0$ and let $\mathscr X$ be a canonical lift of $X$ over $W(k)$ in the sense of Deligne (\cite{DI81}). Let $\mathscr A$ be a canonical lift on $\mathscr X$ of an ample line bundle $A$ on $X$ and let $R$ be the cone with respect to a very ample multiple of $\mathscr A$. Is $R$ \claperfdinj?
\end{question}

\bibliographystyle{skalpha}
\bibliography{MainBib}
\end{document}